\documentclass[a4paper,11pt]{amsart}
\usepackage{amssymb,graphicx,float}
\usepackage{color,xcolor}
 \newtheorem{thm}{Theorem}[section]
 \newtheorem{cor}[thm]{Corollary}
 \newtheorem{lem}[thm]{Lemma}
 \newtheorem{prop}[thm]{Proposition}
 \theoremstyle{definition}
 \newtheorem{defn}[thm]{Definition}
 \theoremstyle{remark}
 \newtheorem{rem}[thm]{Remark}
 \newtheorem{ex}[thm]{Example}
 \numberwithin{equation}{section}

\newcommand{\CC}{\mathbb{C}}
\newcommand{\TT}{\mathbb{T}}
\newcommand{\FF}{\mathbb{F}}
\newcommand{\RR}{\mathbb{R}}
\newcommand{\NN}{\mathbb{N}}

\def\cX{\mathcal{X}}
\def\cV{\mathcal{V}}
\def\cY{\mathcal{Y}}
\def\cZ{\mathcal{Z}}
\def\PP{\mathbb{P}}
\def\Exp{\mathop{\mathrm{Exp}}\nolimits}
\def\Arg{\mathop{\mathrm{Arg}}\nolimits}
\def\Span{\mathop{\mathrm{Span}}}
\def\norm{{\|\cdot\|}}
\newcommand{\Ker}{\mathop{\mathrm{Ker}}\nolimits}

\newcommand{\simbj}{\mathbin{\sim_{BJ}}}
\newcommand\pperp{\protect\mathpalette{\protect\independenT}{\perp}}
    \def\independenT#1#2{\mathrel{\rlap{$#1#2$}\mkern2mu{#1#2}}}

\allowdisplaybreaks
\begin{document}

%
%
%
%
%
%
%
%
%

\title[Birkhoff--James classification of  norm's properties]
 {Birkhoff--James classification of  norm's\\ properties}

\author[A. Guterman]{Alexander~Guterman$^{1}$}
\address{Department of Mathematics, Bar-Ilan University, Ramat-Gan, 5290002, Israel.}
\email{alexander.guterman@biu.ac.il}
\thanks{$^{2,3}$This work is supported in part by the Slovenian Research Agency (research program P1-0285 and research projects N1-0210, N1-0296 and J1-50000). \\ $^{4}$The work was financially supported by RSF grant 22-11-00052.}

\author[B. Kuzma]{Bojan~Kuzma$^{2}$}
\address{$^{a}$University of Primorska, Glagolja\v{s}ka 8, SI-6000 Koper, Slovenia, and
$^{b}$Institute of Mathematics, Physics, and Mechanics, Jadranska 19, SI-1000 Ljubljana, Slovenia.}
 \email{bojan.kuzma@upr.si}

\author[S. Singla]{Sushil~Singla$^{3}$}
\address{University of Primorska, Glagolja\v{s}ka 8, SI-6000 Koper, Slovenia.}
 \email{ss774@snu.edu.in}

\author[S. Zhilina]{Svetlana~Zhilina$^{4}$}
\address{$^{a}$Department of Mathematics and Mechanics, Lomonosov Moscow State University, Moscow, 119991, Russia, and
$^{b}$Moscow Center for Fundamental and Applied Mathematics, Moscow, 119991, Russia.}
\email{s.a.zhilina@gmail.com}


\dedicatory{Dedicated to Professor Chi-Kwong Li, a mentor and a friend,\\ on occasion of his  65's birthday}
\date{}

\begin{abstract}
For an arbitrary normed space $\cX$ over a field $\FF \in \{ \RR, \CC \}$, we define the directed graph $\Gamma(\cX)$ induced by Birkhoff–James orthogonality on the projective space $\mathbb P(\cX)$, and also its nonprojective counterpart $\Gamma_0(\cX)$. We show that, in finite-dimensional normed spaces, $\Gamma(\cX)$ carries all the information about the dimension, smooth points, and norm’s maximal faces. It also allows to determine whether the norm is a supremum norm or not, and thus classifies finite-dimensional abelian $C^\ast$-algebras among other normed spaces. We further establish the necessary and sufficient conditions under which the graph $\Gamma_0(\mathcal{R})$ of a (real or complex) Radon plane~$\mathcal{R}$ is isomorphic to the graph $\Gamma_0(\FF^2, \norm_2)$ of the two-dimensional Hilbert space and construct examples of such nonsmooth Radon planes.
\end{abstract}

\maketitle
\section{Introduction and Preliminaries}


How important is a given relation?
The question is intentionally vague and its answer depends very much on the context, i.e, on the choice of objects of study. One can measure the importance of the relation by using graphs: The vertices, typically (though not always), are all the elements of a given object and two vertices are connected by a directed edge if they are related.  This approach concentrates on the relation alone and leaves aside how the relation interacts with other structures/operations inside the object. Having built the  graphs on two objects within the same category (like the category {\bf Mat} of  matrix algebras),  one can then examine how close the two objects are if the relation on them behaves in exactly the same way. Technically,  if the two graphs are isomorphic, does it imply that the two objects are isomorphic? We call this {\it an isomorphism problem} and if the answer is yes on objects within some category, then the relation is clearly very important there. For example, it was shown in~\cite{Kuzma-JOT} that commutativity has exactly this property in the category ${\bf Mat}(\CC)$  of complex matrix algebras. The same relation  alone can tell if a given nonabelian finite group is simple and can also distinguish among them \cite{Abdollahi,Han-Chen-Guo,Solomon-Woldar}, although it is less powerful on general groups~\cite{Came, Giu-Kuzm}.

Another measure of the importance of the relation is whether it  can classify elements with a particular property inside an object. We  call this \emph{property recognition problem}.
Presently we investigate these questions in the category ${\bf Ban}$ of finite-dimensional Banach spaces. In this paper we focus our attention on the relation of Birkhoff--James orthogonality.

Given a  normed vector space $\cX$ over the field $\FF\in\{\RR,\CC\}$, there are several nonequivalent possibilities to equip it with orthogonality relation. One of the most known and investigated among them is Birkhoff--James orthogonality (BJ-ortho\-go\-na\-lity for short), denoted in the present paper by $x\perp y$. In the words of Birkhoff~\cite{Birkhoff} (choosing the origin to be a point which he denoted by $p$, since Birkhoff considered vectors of the form $\overrightarrow{pq}$), it is defined as a condition that no point on the (extended) line containing $0$ and $y$ is nearer  to~$x$ than $0$.  Equivalently, $x\perp y$ if
$$\|x+\lambda y\|\ge \|x\|;\qquad \lambda \in \FF.$$

Birkhoff then proved in~\cite{Birkhoff} that if any given line has at most one perpendicular to any given point outside of this line, then the above relation on a three-dimensional (real) space is symmetric, i.e.,
$$x\perp y \Longleftrightarrow y\perp x,$$
if and only if the norm is induced by an inner product. The first condition here is interpreted as follows: let $\ell$ be a line passing through some distinct points~$x$ and~$y$, and let~$z$ be an arbitrary point which does not belong to~$\ell$. Then there exists at most one point $w \in \ell$ such that $\overrightarrow{wz} = z - w \perp y - x = \overrightarrow{xy}$ (this is equivalent to the fact that the norm on~$\cX$ is rotund, i.e., strictly convex).

This orthogonality relation was further developed (again in real spaces)  by James in a series of papers~\cite{James,James0,James1}; in particular in \cite[Theorem~1]{James0} he managed to remove the rotundness condition from Birkhoff's result  by a nice application of Kakutani's~\cite[Theorem 3]{Kakutani} classification of inner-product spaces using the norm of a projection. Note that Kakutani's  proof works only in real spaces, however it can be extended without altering its statement to complex normed spaces as well (see~\cite{bohnenblust}). Thus, by \cite[Theorem~1]{James0}, the symmetricity of  BJ-orthogonality is equivalent to the fact that the norm is induced by an inner product, for any  normed space of dimension at least three and over any field $\FF=\RR,\CC$.
The interesting part of this result is that it classifies inner-product spaces purely in terms  of the BJ-orthogonality relation without using additional operations available in every vector space (additivity and multiplicativity by scalars).

A natural tool to investigate this further (i.e., the information on a normed space encoded by the properties of BJ-orthogonality) is via a graph of the relation. More precisely, to every normed space $\cX$ we associate a directed graph (i.e, a digraph) $\Gamma_0=\Gamma_0(\cX)$ whose vertices are all the elements of $\cX$ and vertices $x,y$ form a directed edge $x\rightarrow y$ if $x\perp y$. Observe that $x\perp x$ if and only if $x=0$. Therefore, $x=0$ is the only vertex with a loop in the graph  $\Gamma_0$.

Since Birkhoff--James relation is clearly homogeneous  in a sense that $x\perp y$ if and only if $(\lambda x)\perp (\mu y)$ for every $\lambda,\mu\in\FF$, there is another natural digraph, $\Gamma=\Gamma(\cX)$, associated with $\cX$.
Its vertex set equals $\PP\cX:=\{[x]=\FF x;\;\;x\in\cX\setminus\{0\}\}$ (i.e., a vertex in $\Gamma(\cX)$ is a one-dimensional subspace of~$\cX$), and two vertices $[x],[y]$ form a directed edge (denoted $[x]\rightarrow [y]$) if some, hence any, of their representatives $x\in[x]$ and $y\in[y]$ satisfy $x\perp y$. Observe that $\Gamma(\cX)$ has no loops. Observe also that $\Gamma(\cX)$ is a quotient graph of the induced subgraph of $\Gamma_0(\cX)$ obtained by removing its only loop vertex, under the relation $\FF x=\FF y$ on $\Gamma_0(\cX)\setminus\{0\}$. We call $\Gamma=\Gamma(\cX)$ an \emph{orthodigraph} in the sequel.

Because an orthodigraph was defined over a projectivization rather than a space itself, it encodes one additional piece of information besides the BJ-orthogonality, namely, linear dependence between two vectors. This subtlety is by no means automatic (see~
\cite[Example~3.9]{AGKRZ23}) and may on  the one hand appear rather restrictive, while, on the other hand, makes our key results valid for general norms.

For a subset $\Omega$ we denote by $|\Omega|$  its cardinality. If $S$ is a subset of a (general)  digraph $\hat{\Gamma}$ then $x\rightarrow S$ will be a shorthand for the statement that  $x\rightarrow s$ for every $s\in S$, i.e., $x$ forms a directed edge with every element in $S$. Recall that a clique in digraph $\hat{\Gamma}$ is a subset of pairwise connected vertices.  Given a vertex $x$ in a (general) digraph $\hat{\Gamma}$, we let
$$x^\bot:=\{y\in\hat{\Gamma};\;\; x\rightarrow y\}$$
be its \emph{outgoing neighborhood}, i.e., the set of vertices that have an edge from~$x$. Also,
$${}^\bot x:=\{y\in\hat{\Gamma};\;\; y\rightarrow x\}$$
denotes its \emph{incoming neighborhood}, i.e., the set of vertices that have an edge to~$x$. Notice that for $[x]\in\Gamma(\cX)$ we have $$[x]^\bot=\{[y]\in\Gamma(\cX);\;\;[x]\perp[y]\};$$ similarly for $x^\bot$  with $x\in\Gamma_0(\cX)$. Actually, since BJ-orthogonality is homogeneous,
\begin{equation}
\begin{aligned}
   \relax [z]\in [x]^\bot &\quad\hbox{ if and only if }\quad [z]\subseteq x^\bot & (x\in\cX\setminus\{0\});\\
    z\in x^\bot\setminus\{0\} &\quad\hbox{ if and only if }\quad [z]\in[x]^\bot &(x\in\cX\setminus\{0\}).\label{eq15}
\end{aligned}
\end{equation}
Similar description/connections hold for/between   ${}^\bot x$ and   ${}^\bot[x]$.

Given a subset of vertices $S\subseteq\Gamma_0(\cX)$, its span, $\Span S$, is the set of vertices obtained by taking the linear span of $S$ in the vector space $\cX$.
With our first  result we show that the  operation of taking the spans is   graphological  in smooth normed spaces (that is, the span can be computed by studying the connection in the graph alone and, consequently,  $\Phi(\Span S)=\Span\Phi(S)$ for every bijection $\Phi$ which preserves BJ-orthogonality in both directions). We furthermore show that the  dimension is also a graphological property in a general  normed space.

\begin{thm}\label{thm31} Let $\cX$ be a normed space over $\mathbb F$. Then $\dim\cX<\infty$ if and only if the clique number of  $\Gamma_0(\cX)$ is finite.   If the clique number of  $\Gamma_0(\cX)$ is finite, then
\begin{equation}\label{eq:dim}
    \dim\cX=\min\bigg\{|\Omega|;\; \Omega \text{   a finite subset of } \Gamma_0(\cX)\hbox{ with } \bigg|\bigcap_{x\in\Omega} x^\bot\bigg| = 1\bigg\}.
\end{equation}
Furthermore, if $\cX$ is a smooth finite-dimensional space and $S$ is a subset of $\Gamma_0(\cX)$, then
\begin{gather*}
\mathrm{span}(S)=\bigcap_{\substack{x\in\Gamma_0(\cX)\\ x\rightarrow S}} x^\bot \quad \text{and}\\
\mathrm{dim}(\mathrm{span}(S))=\dim\cX-\min\bigg\{|\Omega|;\; \Omega\subseteq\Gamma_0(\cX)\text{ with }\bigcap_{\substack{x\in\Omega\\ x\rightarrow S}} x^\bot=\mathrm{span}(S)\bigg\}.
\end{gather*}
\end{thm}

\begin{rem}\label{remark1}

The proof will show that both formulas for the dimension inside Theorem \ref{thm31} also hold if $\Gamma_0(\cX)$ (which contains exactly one loop vertex) is replaced by $\Gamma(\cX)$ (which contains no  loop vertices), the only difference is that in \eqref{eq:dim} we minimize the cardinality of subsets $\Omega\subseteq\Gamma(\cX)$ with $\left|\bigcap_{x\in\Omega} x^\bot\right|=0$. In this light we can define the dimension, $\dim\hat{\Gamma}$, of  any digraph $\hat{\Gamma}$, as follows:
\begin{enumerate}
    \item[(i)] $\dim\hat{\Gamma}=\infty$ if $\hat{\Gamma}$  contains an infinite clique,
    \item[(ii)] $\dim\hat{\Gamma}$ is given by \eqref{eq:dim} if $\hat{\Gamma}$ has a finite clique number and a loop vertex,
    \item[(iii)] $\dim\hat{\Gamma}=\min\left\{|\Omega|;\; \Omega \text{ is  a finite subset of } \hat{\Gamma}(\cX)\hbox{ with } \left|\bigcap_{x\in\Omega} x^\bot\right|=0\right\}$ if $\hat{\Gamma}$ is loopless with a finite clique number.
\end{enumerate}
(We follow the convention that $\min\varnothing =\infty$.)
\end{rem}

Recall that a nonzero point $x$ in a vector space $\cX$ is smooth if it has a unique supporting functional (a normalized functional $f$ is supporting for $x$ if $f(x)=\|x\|$). It is easy to see that  $x$ is smooth if and only if $\lambda x$ is smooth for every nonzero $\lambda\in\FF$. With this in mind, we call a point $[x]$ in a projective space $\PP\cX$ to be smooth if some, hence any, of its representatives is a smooth point in $\cX$.

A face of the norm's unit ball is a convex subset $F$ on the norm's unit sphere such that if $x,y$ are different normalized vectors whose   open line segment $\{\lambda x+(1-\lambda y);\; 0<\lambda <1\}$ intersects $F$, then already $x,y\in F$.
We define $\mathcal F\subseteq\PP(\cX)$ to be a face if we can choose a representative for each $[x]\in\mathcal F$ such that $\{x;\; [x]\in\PP(\cX)\}$ is a face of $\cX$. A face is maximal if it is not properly contained in a bigger face.

Given a (nonzero) point in a normed space, can we determine if it is smooth from BJ-orthogonality relation alone? Such a point corresponds to some  vertex in orthodigraph and we can equivalently ask: Can we determine whether this vertex comes from a smooth point by studying only its connections within the digraph? We give an affirmative answer in Lemmas~\ref{lemma:two-dimensional-projective} and~\ref{lemma:high-dimensional-projective} below but the answer is rather technical (likewise we show in Corollary \ref{cor:maximalfaceclassify} below how to collect all vertices that correspond to a maximal face).  We temporarily content ourselves  with a slightly less precise corollary of this classification: Every bijection between the projectivizations of normed spaces~$\cX$ and~$\cY$ which strongly preserves BJ-orthogonality (that is, every isomorphism  between orthodigraphs) must preserve smooth points  and maximal faces. Recall that the dimension of orthodigraph was defined within Remark~\ref{remark1}.

\begin{thm}\label{thm41}
    Let $\cX, \cY$ be normed spaces over $\FF$. If orthodigraph $\Gamma(\cX)$ is finite-dimensional  then an orthodigraph isomorphism $\Gamma(\cX)\to\Gamma(\cY)$ maps smooth points onto smooth points and  maps maximal faces onto maximal faces.
    Furthermore, $\cY$ is finite-dimensional and $\dim\cX=\dim\cY$.
\end{thm}

Our final main result allows us to distinguish the supremum norm on a finite-dimensional space from any other norm in terms of $\Gamma(\cX)$ alone.

\begin{thm}\label{main3} Let $(\cX, \norm)$ be a finite-dimensional  normed space over $\mathbb F$. Then the following conditions are equivalent:
  \begin{itemize}
    \item[(i)] There exists a linear bijection $A\colon \cX\to\FF^n$  such that $\|x\|=\|Ax\|_{\infty}$ for some~$n\in\NN$.
    \item[(ii)] The cardinality of $\{x^\bot;\;\;x\in\Gamma(\cX),\hbox{ and $x$ is smooth}\}$ is minimal possible among all the norms on $\cX$.
    \item[(iii)] $|\{x^\bot;\;\;x\in\Gamma(\cX),\hbox{ and $x$ is smooth}\}|=\dim\cX$.
    \end{itemize}
\end{thm}

This has the following immediate application to orthodigraphs of fi\-nite-di\-men\-sional commutative $C^*$-algebras. It is well-known (\cite[II.2.2.4 and II.1.1.3.(2)]{Black}) that any commutative $C^*$-algebra $\mathcal{A}$ is $*$-isomorphic to $C_0(\Omega)$, the $C^*$-algebra of all continuous complex functions on a locally compact Hausdorff space $\Omega$ vanishing at `infinity', with the norm $\|g\| = \max \{|g(t)|; \; g\in \Omega\}$. If $\mathcal{A}$ is finite-dimensional, then $\Omega$ is finite, so $\mathcal{A}$ is isomorphic to $(\CC^n, \norm_{\infty})$ with componentwise multiplication for some $n \in \mathbb{N}$. Hence, Theorem~\ref{main3} solves the problem of classification of commutative $C^\ast$-algebras among general finite-dimensional normed spaces.

We remark that Ryotaro Tanaka~\cite[Proposition 4.5, Theorem~4.6, and Corollary 4.7]{Tan521} has recently obtained a similar classification of supremum norm based on topological properties of ultrafilters associated with  outgoing neighborhoods of $\Gamma_0(\cX)$. Moreover, he has obtained~\cite{Tanaka-personal} another formula  for computing the dimension, which does not directly use (ortho)digraphs as does our formula in \eqref{eq:dim}, but  is  again based on topological properties of ultrafilters associated with $
\Gamma_0(\cX)$. We are indebted to Professor Tanaka for sending us his preprint. A different approach to compute the dimension,  based on BJ-orthogonality in a projective space, can be found in~\cite[Lemma~2.3 and Remark~2.4]{AGKRZ23}.

\medskip

The proofs of our main results will be given in Section \ref{section3}. We now summarize the rest of the results of our paper where we study how far Theorems~\ref{thm41} and~\ref{main3} can be extended from the projective orthodigraph, $\Gamma(\cX)$, to its nonprojective counterpart, $\Gamma_0(\cX)$. For this purpose, in Section~\ref{section:BJ-norms}, we introduce the class of BJ-normed spaces. In particular, if the norm on $\cX$ is rotund or smooth, then $\cX$ is BJ-normed. However, as Example~\ref{example:BJ-norm} shows, the converse is not true. By Corollaries~\ref{corollary:two-dimensional} and~\ref{corollary:high-dimensional}, in BJ-normed spaces smooth points can be distinguished from the nonsmooth ones even in the nonprojective setting, so Theorem~\ref{thm41} remains valid for an isomorphism $\Gamma_0(\cX) \to \Gamma_0(\cY)$. Besides, as shown in Lemma~\ref{lemma:isomorphism-implication}, if $\cX$ and $\cY$ are BJ-normed spaces, then any isomorphism between $\Gamma_0(\cX)$ and $\Gamma_0(\cY)$ induces an isomorphism between $\Gamma(\cX)$ and $\Gamma(\cY)$.

However, these results are no longer valid for arbitrary normed spaces. Recall that a two-dimensional normed space where BJ-othogonality is symmetric is called a Radon plane. In Lemmas~\ref{lem:nnnsmoothRadonplaneRR} and~\ref{lem:nnnsmoothRadonplaneCC}, we establish the necessary and sufficient conditions under which the graph $\Gamma_0(\mathcal{R})$ of a (real or complex) Radon plane~$\mathcal{R}$ is isomorphic to the graph $\Gamma_0(\FF^2, \norm_2)$ of the two-dimensional Hilbert space. In Examples~\ref{example:nonsmooth-Radon-plane} and~\ref{example:complex-Radon-plane}, we construct a nonsmooth Radon plane~$\mathcal{R}$ over $\FF \in \{ \RR, \CC \}$ such that $\Gamma_0(\mathcal{R})$ and $\Gamma_0(\FF^2, \norm_2)$ are isomorphic. Notice, in contrast, that for a nonsmooth Radon plane $\mathcal R$, Theorem~\ref{thm41} implies that $\Gamma(\mathcal{R})$ and $\Gamma(\FF^2, \norm_2)$ are not isomorphic, since the norm on $(\FF^2, \norm_2)$ is smooth. Next, in Example~\ref{example:nonsmooth-point}, we show that the assumption about a BJ-normed space is essential in Corollary~\ref{corollary:high-dimensional}. In Section~\ref{section5}, we end with some related results and concluding remarks.

\section{Proofs of main results}\label{section3}
It will be beneficial to utilize the norm's subdifferential at a vector $x$, denoted by $\partial\|x\|$ (see \cite[Definition~1.9]{Phelps} and \cite[Definition 1.2.4]{Grover}), which is defined as
$$\partial\|x\|=\{f\in\cX^*; \; \|y\|-\|x\|\geq \mathrm{Re} \left(f(y-x)\right)\text{ for all }y\in\cX\}.$$
We note that the subdifferential set at a nonzero vector $x$ coincides with the set of all supporting functionals at $x$ (see \cite[Lemma~5.10]{Phelps} and \cite[Example 1.2.16]{Grover}). By~\cite[Theorem 2.1]{James1} (which is stated for real spaces but the proof works for complex ones as well), we have $x \perp y$ if and only if there exists a supporting functional~$f$ at~$x$  which annihilates~$y$. Hence
\begin{equation}\label{eq36}
x^\bot =
\bigcup_{f\in\partial\|x\|}\Ker f.
\end{equation}

Recall that, if $\cX$ is a complex normed space with $\dim \cX = n$, by restriction of scalars we may regard it as a $2n$-dimensional real normed space, which, to avoid the confusion, we denote as $(\cX_{\RR},\norm_{\RR})$. If $\cX$ is a real normed space, then we set $\cX_{\RR} = \cX$. Then a face $F$ of the norm's unit ball in $\cX$ is called {\em exposed} if
$$F=\{x\in\cX;\;\;\|x\|=1 \hbox{ and } f(x)=\|f\|\}$$ for some $\RR$-linear functional $f$ on $\cX$, i.e., $f\in \cX_{\RR}^*$, the dual space of $\cX_{\RR}$.

\begin{lem} \label{proposition:exposed-faces}
Let $\cX$ be a normed space over $\FF \in \{ \RR, \CC \}$. The face $F$ of the norm's unit ball in $\cX$ is exposed if and only if
\begin{equation} \label{eq:exposed-face-def}
    F=\{x\in\cX;\;\;\|x\|=1 \hbox{ and } f(x)=\|f\|\}
\end{equation}
for some $\FF$-linear functional $f$ on $\cX$, i.e., $f\in \cX^*$.
\end{lem}

\begin{proof}
If $\FF = \RR$, then the statement is immediate from the definition of an exposed face, so let $\FF = \CC$. Assume first that $F$ is an exposed face, and let $f$ be the corresponding $\RR$-linear functional. Then it is a real part of a $\CC$-linear functional $\widetilde{f}(x):=f(x)-if(ix)$ which has the same norm as $f$. For any $x \in F$, we have $\| f \| = \| \widetilde{f} \| \ge|\widetilde{f}(x)|=|f(x)-i f(ix)|=\sqrt{|f(x)|^2+|f(ix)|^2}\ge |f(x)| = \| f \|$, so $f(ix)=0$ and $\widetilde{f}(x) = f(x) = \| \widetilde{f} \|$. Thus $F=\{x\in\cX;\;\|x\|=1 \hbox{ and } \widetilde{f}(x)=\|\widetilde{f}\|\}.$

Conversely, let $F=\{x\in\cX;\;\|x\|=1 \hbox{ and } f(x)=\|f\|\}$ for some $\CC$-linear functional $f$ on $\cX$. By~\cite[p.~58]{RUD}, its real part $(\mathrm{Re}\,f)$ is an $\RR$-linear functional which satisfies $\|(\mathrm{Re}\,f)\|=\|f\|$. Therefore, $F=\{x\in\cX;\;\|x\|=1 \hbox{ and } (\mathrm{Re}\,f)(x)=\|(\mathrm{Re}\,f)\|\}$, so $F$ is an exposed face.
\end{proof}

Recall that an {\em exposed point} of the norm's unit ball is a singleton exposed face. Note that, for an exposed point $x$, the functional $f$ in \eqref{eq:exposed-face-def} may not be unique, so there might exist two distinct functionals $f$ and $g$ of unit norm such that $$\{x\}=\{x\in\cX;\;\;\|x\|=1 \hbox{ and } f(x)=g(x)=1\}.$$
The smooth points of a normed space are related to the exposed points of the dual space as follows.

\begin{lem}\label{exposed-smooth}
   Let $(\cX,\norm)$ be a finite-dimensional normed space over $\FF$. The set of supporting functionals of smooth points in~$\cX$ coincides with the set $\Exp(\cX^\ast)$ of exposed points in the dual norm.
\end{lem}
\begin{proof}
    Let $x\in\cX$ be a smooth point and  $f$ be its  unique supporting functional. Let us consider the evaluation functional $\varphi_x$ defined on $\cX^*$ by $\varphi_x(g) = g(x)$. Note that, for $g$ of norm one, we have $\varphi_x(g)=1$ if and only $g$ is a supporting functional of $x$. Since $f$ is the unique supporting functional of $x$, we have
    $$\{f\}=\{g\in \cX^*;\; \|g\|= \varphi_x(g)=1\}.$$ So, $f$ is an exposed point of $\cX^*$.

    Conversely, let $f$ be an exposed point of $\cX^*$. Then there exists a functional $\varphi$ on $\cX^*$ of norm one such that $\{f\}=\{g\in\cX^*;\; \|g\|=\varphi(g)=1\}.$ Since $\cX$ is reflexive, there exists $x\in \cX$ such that $\varphi=\varphi_x$. It means $$\{f\}=\{g\in\cX^*;\; \|g\|=\varphi_x(g)=g(x)=1\}.$$ So, $f$ is the unique supporting functional for $x$. Hence, $x$ is a smooth point whose supporting functional is $f$.
\end{proof}

We note that, in any  finite-dimensional normed space $\cX$ over $\FF$, the convex span of its exposed points is dense in its closed unit ball. For  $\FF=\RR$, the proof can be found in \cite[Theorem 18.7]{Rockafellar}. For the case $\FF=\CC$, we note that the set of exposed points in $\cX$ coincides with the set of exposed points in $\cX_\RR$, and then the real case can be applied.

\begin{cor}\label{lem:lin-indep-smooth}
  Let $(\cX,\norm)$ be a finite-dimensional normed space over $\FF$. Then~$\cX$ has at least $n=\dim\cX$ smooth points whose supporting functionals are linearly independent.
  \end{cor}
\begin{proof} As shown above, the convex span of $\Exp(\cX^\ast)$, the set of all exposed points in a dual space, is dense in the closed unit ball. So $\cX^\ast$ has a basis consisting of exposed points $f_1,\dots ,f_n\in\Exp(\cX^\ast)$. By Lemma \ref{exposed-smooth}, these are the  supporting functionals of some  smooth points $x_1,\dots,x_n$.
\end{proof}

\begin{rem}
  Note that this has no counterpart in infinite-dimensional spaces. In fact, the Banach space $\ell_1(\Lambda)$  of summable sequences over an uncountable index set~$\Lambda$ has no smooth points; see \cite[Example 1.4(b)]{Phelps}.
\end{rem}

\begin{cor}\label{inequality}  Let $\cX$ be a finite-dimensional  normed space over $\mathbb F$. Then
\begin{equation}\label{89}\left|\{x^\bot ;\; x\in\Gamma(\cX), \text{ and } x\text{ is smooth}\}\right|\geq \mathrm{dim}(\cX).\end{equation}
\end{cor}

\begin{proof}
    By Corollary~\ref{lem:lin-indep-smooth}, there exist  $n=\dim\cX$ smooth points $x_1,\dots,x_n\in\cX$ whose supporting functionals $f_1,\dots,f_n$ form a basis of the dual space $\cX^\ast$, and therefore, their kernels are pairwise distinct. Since each of the points $x_i$ is smooth, we have $x_i^\bot=\Ker f_i$. Hence, if $[x_i]$ denotes the corresponding projective point, then $[x_1]^\bot,\dots,[x_n]^\bot$ are also pairwise distinct.
\end{proof}

\begin{cor}\label{cor:dim}
  If $\cX$ contains  infinitely many pairwise BJ-orthogonal  vectors, then $\dim\cX=\infty$.
  Otherwise, $\dim\cX$ equals the cardinality  of the smallest subset $\Omega\subseteq\cX$ such that $x\perp z$ for every $x\in \Omega$ implies $z=0$, that is,
  $$\dim\cX=\min_{\substack{\Omega\subseteq\cX \\ \bigcap \limits_{x\in \Omega} x^\bot=\{0\}}}|\Omega|.$$
\end{cor}
\begin{proof}
  The infinite-dimensional case is covered in \cite[Corollary 3.1 and Theorem 3.5]{AGKR-symmetrized}.
  Assume next $n:=\dim\cX<
\infty$. By Corollary~\ref{lem:lin-indep-smooth}, there exist $n$ smooth points $x_1, \dots, x_n\in\cX$ such that their (unique) supporting functionals $f_{x_1},\dots, f_{x_n}$ are linearly independent. Hence, $$\bigcap_{i=1}^n x_i^\bot =\bigcap_{i=1}^n \Ker f_{x_i} = \{0\}.$$ By letting $\Omega=\{x_1,\dots,x_n\}$ we see that $$
\min_{\substack{\Omega\subseteq\cX \\ \bigcap \limits_{x\in \Omega} x^\bot=\{0\}}}|\Omega|
\leq n.$$
Conversely,  consider any $\Omega=\{x_1,\dots, x_k\}\subseteq \cX$ with $|\Omega|=k<n$, and for each $x_i$ let  $f_{x_i}$ be one of its supporting functionals. Then $\Ker f_{x_i}\subseteq x_i^\bot$ for all $1\leq i\leq k$. Since $k<n$ we have, by subadditivity of codimension on intersections of spaces, $\{0\}\neq \bigcap_{i=1}^{k} \Ker f_{x_i}\subseteq \bigcap_{i=1}^{k}x_i^\bot$,  which proves that
$$\min_{\substack{\Omega\subseteq\cX \\ \bigcap \limits_{x\in \Omega} x^\bot=\{0\}}}|\Omega|\geq n,$$
and  the result follows.
\end{proof}

\begin{proof}[Proof of Theorem~\ref{thm31}] Notice that $z\in\bigcap_{x\in\Omega} x^\bot$ implies $\FF z\subseteq \bigcap_{x\in\Omega} x^\bot$, so this  intersection is either infinite or else it has exactly one element.  Hence, \eqref{eq:dim} follows directly from  Corollary~\ref{cor:dim}.

If one replaces digraph $\Gamma_0(\cX)$ by its projective counterpart  $\Gamma(\cX)$,  then the intersection  $\bigcap_{[x]\in\Omega}[x]^\bot$ is either empty or else it contains at least one  point in projective space. The former case is equivalent to $\bigcap_{x\in\Omega_0} x^\bot=\{0\}$, where $\Omega_0\subseteq\cX$ is the  collection of vectors obtained by taking a single representative $x$ from each projective point $[x]\in\Omega$  (thus, $|\Omega_0|=|\Omega|$). This validates  the claim in  Remark~\ref{remark1}.

As for the second part, let  $\cX$ be a smooth space with $\dim \cX=n$ and let $\dim(\mathrm{span}(S))=k$. Recall~\cite[Theorem 4.2]{James1} that, in smooth spaces, BJ-orthogonality is right additive (note that~\cite{James1} deals with real normed spaces only but the proof of \cite[Theorem 4.2]{James1} for complex normed spaces follows along the same lines). Hence $x\rightarrow S$ is equivalent to $x\rightarrow (\Span S)$ as well as to $x^\bot\supseteq\Span S $. There exist exactly $n-k$ linearly independent functionals $f_1,\dots ,f_{n-k}$ which annihilate $\Span S$. Without loss of generality assume $\|f_i\|=1$ and let $x_1,\dots, x_{n-k}$ be the normalized vectors such that $f_i(x_i)=1$ for all $1\leq i\leq n-k$.
Hence, $f_i$ is a supporting functional for $x_i$ so, being in a smooth space, we have $x_i^\bot=\Ker f_i$. It follows that $\bigcap\limits_{i=1}^{n-k} x_i^\bot=\bigcap\limits_{i=1}^{n-k}\Ker f_i=\Span S$. This gives
$$\Span(S)= \bigcap\limits_{i=1}^{n-k} x_i^\bot\supseteq \bigcap_{\substack{x\in\Gamma_0(\cX)\\ x\rightarrow S}} x^\bot\supseteq \Span(S),$$
and hence $\bigcap\limits_{\substack{x\in\Gamma_0(\cX)\\ x\rightarrow S}} x^\bot=\mathrm{span}(S).$ It also implies that $$\min\bigg\{|\Omega|;\; \Omega\subseteq\Gamma_0(\cX)\text{ with }\bigcap_{\substack{x\in\Omega\\ x\rightarrow S}} x^\bot=\mathrm{span}(S)\bigg\}\leq n-k.$$ Similarly to the proof of Corollary~\ref{cor:dim}, we get the reversed inequality, and the result follows.
\end{proof}
We now show how, on a projective space,  BJ-orthogonality  alone knows which lines are smooth and which are not. This will be done separately for two-dimensional spaces and for spaces of dimension at least three.  Recall from Remark~\ref{remark1}
that the dimension can be computed with the help of BJ-orthogonality alone. For simplicity,  we will typically denote the points in $\Gamma(\cX)$ simply by $x,y,\dots$ instead of $[x],[y],\dots$.

\begin{lem} \label{lemma:two-dimensional-projective}
  Let $\cX$ be a two-dimensional real or complex normed space. Then the following conditions are equivalent:
  \begin{itemize}
    \item[(i)] A vertex $x\in\Gamma( \cX)$ corresponds to  a smooth point in $\PP\cX$.
    \item[(ii)] $|x^\bot|=1$.
  \end{itemize}
\end{lem}
\begin{proof}
Recall that $x^\bot$ coincides with a collection of all projective points that are annihilated by supporting functionals of some representative~$\widehat{x}\in x$. By  definition, two different supporting functionals of  $\widehat{x}$ are linearly independent, so they have different kernels, and, in two-dimensional spaces,
each nonzero functional annihilates a unique projective point.
Hence $|x^\bot|$ is  equal to the cardinality of different supporting functionals of a representative $
\widehat{x}\in x$. Therefore, the result follows because $x$   is nonsmooth  if and only if its representative has at least two supporting functionals.
\end{proof}

\begin{lem} \label{lemma:high-dimensional-projective}
Let $\cX$ be a finite-dimensional normed  space over $\FF$ with {$n = \dim\cX\ge 3$}. Then the following conditions are equivalent:
\begin{itemize}
  \item[(i)] A vertex $x\in\Gamma( \cX)$ corresponds to  a smooth point in $\PP\cX$.
  \item[(ii)] There exist  $(n-2)$ vertices $x_3,\dots,x_{n}\in\Gamma(\cX)$ such that
    $$
    |x^\bot\cap x_3^\bot\cap\dots\cap x_{n}^\bot|=1.
    $$
  Moreover, for every reordering of this $(n-2)$-tuple of vertices we can define $$\Omega:=\{y\in\Gamma(\cX);\;\;
  |x^\bot\cap y^\bot\cap x_4^\bot\cap \dots\cap x_{n}^\bot|=1\},$$ and then $x_4^\bot\cap\dots\cap x_n^\bot\cap x^\bot\cap\bigcap\limits_{y\in\Omega}y^\bot=\varnothing$.
\end{itemize}
\end{lem}

\begin{proof}

 (i) $\Rightarrow$ (ii).
By Corollary~\ref{lem:lin-indep-smooth}, there exist $n$ smooth points $\widehat{x}_1,\dots,\widehat{x}_n\in\cX$
whose supporting functionals $f_{1},\dots,f_{n}$ form a basis for the dual space; let $x_i:=\FF \widehat{x}_i\in\PP\cX$ be the corresponding projective points. Then
\begin{equation}\label{eq:x^bot-vs-Ker(f_x)}
    x_i^\bot=\{z\in\PP\cX;\;z\subseteq\Ker f_{i}\}.
\end{equation}
 Let  $x\in\PP\cX$ be a smooth point and $f_x$ be a unique supporting functional of some representative $\widehat{x}\in x$. After reordering $f_{1},\dots,f_{n}$ we can assume that
 $$f_x,f_{2},\dots,f_{n}$$ form a basis for the dual space $\cX^\ast$. Choose now any reordering of $\widehat{x}_3,\dots,\widehat{x}_n$; for simplicity we consider only the identical reordering.
 Then
 $$\dim\bigg(\Ker f_x\cap \Ker f\cap\bigcap_{i=4}^n\Ker f_{i} \bigg)=1,\qquad f\in\{f_{2},f_{3}\};$$
and since  $f_x$, $f_{i}$ are the only supporting functionals at the smooth points $\widehat{x}$ and $\widehat{x}_i$, respectively, we have, by \eqref{eq:x^bot-vs-Ker(f_x)}:
$$
   \biggl|x^\bot\cap w^\bot\cap \bigcap_{i=4}^n x_i^\bot \biggr|=1,\qquad w\in\{x_2,x_3\}.
$$
As such, $x_2,x_3\in\Omega$. Since $f_x,f_{2},\dots,f_{n}$ form a basis in $\cX^*$, we clearly have
\begin{gather*}
    \bigcap_{i=4}^n\Ker f_{i} \cap\Ker f_x\cap (\Ker f_{2}\cap\Ker f_{3}) =0, \quad \text{so}\\
    x_4^\bot\cap\dots\cap x_n^\bot \cap x^\bot\cap\bigcap_{y\in\Omega} y^\bot\subseteq x_4^\bot\cap\dots\cap x_n^\bot \cap x^\bot\cap x_2^\bot\cap x_3^\bot=\varnothing,
\end{gather*}
as claimed.

$\neg$ (i) $\Rightarrow$ $\neg$ (ii).  Let $x$ be nonsmooth. Then its representative $\widehat{x}\in x$  has at least two $\FF$-linearly independent supporting functionals $f_1,f_2$. Consider now any set of $(n-2)$ projective points $x_3,\dots,x_{n}\in\Gamma(\cX)$ which satisfies
\begin{equation} \label{equation:one-intersection}
    |x^\bot\cap x_3^\bot\cap\dots\cap x_{n}^\bot|=1;
\end{equation}
if no such $n-2$ tuple exists, then there is nothing to do.

Let $f_i$ be an arbitrary supporting functional at a representative $\widehat{x}_i\in x_i$, $i = 3, \dots, n$. Then $f_1, f_3, \dots, f_n$ are linearly independent, since otherwise $\Ker f_1 \cap \Ker f_3 \cap \dots \cap \Ker f_n$ would be at least two-dimensional subspace  with every line inside it  corresponding to a  projective point contained in $x^\bot\cap x_3^\bot\cap\dots\cap x_{n}^\bot$, contradicting \eqref{equation:one-intersection}. Similarly, $f_2, f_3, \dots, f_n$ are linearly independent. However, $f_1, f_2, f_3, \dots, f_n$ are linearly dependent, since otherwise $\Ker f_1 \cap \Ker f_3 \cap \dots \cap \Ker f_n$ and $\Ker f_2 \cap \Ker f_3 \cap \dots \cap \Ker f_n$ would correspond to two distinct projective points contained in $x^\bot\cap x_3^\bot\cap\dots\cap x_{n}^\bot$, again contradicting \eqref{equation:one-intersection}. Since $f_1$ and $f_2$ are linearly independent, without loss of generality we may assume that $f_3$ can be expressed through a linearly independent set $f_1, f_2, f_4, \dots, f_n$.

Consider now an arbitrary $y \in \Omega$, i.e.,
$$
  |x^\bot\cap y^\bot\cap x_4^\bot\cap\dots\cap x_n^\bot|=1,
$$
and an arbitrary supporting functional $f$ at some representative of $y$. Arguing as  before, the functionals $f_1, f_2, f, f_4, \dots, f_n$ are  linearly dependent, so $f$ can be expressed as a linear combination of $f_1, f_2, f_4, \dots, f_n$. Hence, there exists a projective point $z\in\PP\cX$ with
$$
z\subseteq\bigcap_{i \neq 3} \Ker f_i = \Ker f \cap \bigcap_{i \neq 3} \Ker f_i $$
and every line in the latter intersection clearly  corresponds to a point inside $ x^\bot\cap y^\bot\cap x_4^\bot\cap\dots\cap x_n^\bot$ (by \eqref{eq:x^bot-vs-Ker(f_x)}).
Thus, for any $y \in \Omega$,
$$z\in x^\bot\cap x_4^\bot\cap\dots\cap x_n^\bot\cap\bigcap_{y\in\Omega} y^\bot, $$
as claimed.
\end{proof}

Note that \eqref{equation:one-intersection} in the proof above is possible even with a nonsmooth point~$x$. Say, if $\partial\|x\|=\{(1,\lambda,0);\;\;0\le\lambda\le 1\}$, and if $x_1$ is a smooth point with $\partial\|x_1\|:=(1,2,0)$, then
$x^\bot\cap x_1^\bot$  contains a single line $[e_3]$. Moreover, it is also possible that in \eqref{equation:one-intersection} not all $x_i$ are smooth. Say, if $\partial\|x\|=\{(1,\lambda,0,0);\;\;0\le\lambda\le1\}$ and $\partial\|x_1\|=\{(0,\lambda,1,0);\;\;0\le\lambda\le 1\}$ and $\partial\|x_2\|=(0,1,0,0)$. Then $x^\bot\cap x_1^\bot\cap x_2^\bot$ contains just $[(0,0,0,1)]$.

\begin{lem}\label{thm5} Let $\cX$ be a normed space over $\FF$ and let $\{x_j;\; j\in\mathcal J\}$ be a  collection of projective points in  $\PP\cX$. Then the following conditions are equivalent.
\begin{itemize}
\item[(i)] There exists a projective point $z\in\PP\cX$ such that $z^\bot\subseteq\bigcap_{j\in\mathcal J}x_j^\bot$.
\item[(ii)] There exist representatives $x_j'$ of $x_j$ such that the convex hull of $\{x_j';\;j\in\mathcal J\}$ lies in a common face of the norm's  unit sphere.
\end{itemize}
\end{lem}
\begin{proof}  (i) $\Longrightarrow$ (ii).
Let $\widehat{z}$ and $\widehat{x}_j$ be norm-one representatives of  $z$ and $x_j$, respectively. Clearly, $z^\bot\subseteq \bigcap_{j\in\mathcal J}x_j^\bot$ implies $\widehat{z}^\bot\subseteq \bigcap_{j\in\mathcal J}\widehat{x}_j^\bot$. Therefore, by \eqref{eq36}, every supporting functional $f$ of $\widehat{z}$ satisfies $\Ker f\subseteq \bigcap_{j\in\mathcal J}\widehat{x}_j^\bot$.

Now, fix a supporting functional $f$ of $\widehat{z}$. Then, using \cite[Theorem 2.1]{James1}, we get $|f(\widehat{x}_j)|=1$ for all $j\in\mathcal J$. It means there exists $\alpha_j$ with $|\alpha_j|=1$ such that $f(\widehat{x}_j) =\alpha_j$ for all $j\in\mathcal J$. Considering $x_j'=\widehat{x}_j/\alpha_j$, we get $f(x_j')=1$ for all $j\in\mathcal J$. Hence, all $x_j'$ belong to a common face
$$F=\{\widehat{x};\;\|\widehat{x}\|=1\text{ and } f(\widehat{x})=1\}.$$

 (ii) $\Longrightarrow$ (i). Let there exist representatives  $x_j'$ of $x_j$ such that the convex hull of $\{x_j';\; j\in\mathcal J\}$, say $C$, lies on a common face $F$ of the norm's closed unit ball. In particular, for any $\widehat{y}\in C$, we have $\|\widehat{y}\|=1$. By \cite[Theorem~6.2]{Rockafellar}, the relative interior of $C$ is nonempty. Let $\widehat{z}$ be in the relative interior of~$C$ and let $f\in\partial\|\widehat{z}\|$ be any of its supporting functionals. We claim that
 $$\Ker  f\subseteq \bigcap_{j\in\mathcal J}(x_j')^\bot.$$ To see this, fix an index $j\in\mathcal J$. Using \cite[Theorem 6.4]{Rockafellar}, we obtain that there exists $0<\epsilon<1$ such that $\|(1-t)x_j'+t\widehat{z}\|=1$ for all $t\in [0, 1+\epsilon]$. Then $$\big|f(x_j')(1-t)+t\big| = \big|f((1-t)x_j'+t\widehat{z})\big|\leq \|(1-t)x_j'+t\widehat{z}\|=1$$  for all $t\in[1-\epsilon,1+\epsilon]$. In particular, with $t=1\pm \epsilon$, we get $\big|1\pm\epsilon(f(x_j')-1)\big|\leq 1$, which gives $f(x_j') =1$.
So, $f$ is a supporting functional for $x_j'$ also. By \eqref{eq36}, \begin{equation*}\label{extraequation}\Ker f\subseteq (x_j')^\bot.\end{equation*} Since $f\in\partial\|z\|$ was an arbitrary, using \eqref{eq36} again, we get
\begin{equation}\label{eq78}
    z^\bot\subseteq \bigcap_{j\in\mathcal J}(x_j')^\bot\qedhere
\end{equation}
\end{proof}

\begin{rem} \label{substitute:lemma2.10}
   The proof of Lemma \ref{thm5} also shows the following three statements.

\begin{itemize}
\item[(i)] Let $\{x_j;\; j\in\mathcal J\}$ be a collection of projective points in $\PP\cX$. Then $\bigcap_{j\in\mathcal J}x_j^\bot$ contains $z^\bot$ for some $z\in \cX$ if and only if there exists representatives $x_j'$ of $x_j$ such that all the elements in the convex hull of $\{x_j';\;j\in\mathcal J\}$ are of norm one.
\item[(ii)] We have a stronger statement than (i) $\Longrightarrow$ (ii) in Lemma \ref{thm5}: If there exists a linear functional $f$ on $\mathcal X$ such that $\Ker f \subseteq \bigcap_{j\in\mathcal J}x_j^\perp$, then there exist representatives $x_j'$ of $x_j$ such that the convex hull of $\{x_j';\;j\in\mathcal J\}$ lies in a common face of the norm's  unit sphere.
\item[(iii)] If $C$ is a face of norm's closed unit ball and $z_1$ and $z_2$ are in the relative interior of $C$, then $z_1^\bot=z_2^\bot$ (since \eqref{eq78} holds for an arbitrary point $z$ in the relative interior of $C$ and since we may take $\{x_j; \; j \in \mathcal{J}\}=C$).
\end{itemize}\end{rem}

We recall that $\mathcal F\subseteq\PP(\cX)$ is a face (or an exposed face) if we can choose a representative for each $[x]\in\mathcal F$ such that $\{x;\; [x]\in\PP(\cX)\}$ is a face (or an exposed face, respectively) of $\cX$. As an immediate corollary of Lemma \ref{thm5}, we get the  characterization of maximal faces in a projective space in terms of the orthodigraph.
\begin{cor}\label{cor:maximalfaceclassify} A collection of vertices  $\mathcal F\subseteq\Gamma(\cX)$ forms  a maximal face of $\PP\cX$ if and only if the following conditions are satisfied:
\begin{itemize}
\item[(i)] $\bigcap_{x\in\mathcal F} x^\bot$ contains $z^\bot$ for some $z\in \Gamma(\cX)$,
\item[(ii)] $\mathcal F\cup\{y\}$ does not satisfy (i) for any $y\in\Gamma(\cX)\setminus \mathcal F$.
\end{itemize}
\end{cor}

Finally, we provide proofs of our main theorems.
\begin{proof}[Proof of Theorem \ref{thm41}]
The proof that $
\dim\cX=\dim\cY$ follows from Theorem~\ref{thm31} and Remark~\ref{remark1}. The proof of the fact that an orthodigraph isomorphism $\Gamma(\cX)\rightarrow\Gamma(\cY)$ maps smooth points onto smooth points then follows from Lemma~\ref{lemma:two-dimensional-projective} and~\ref{lemma:high-dimensional-projective}.  Finally, the proof of the fact that an orthodigraph isomorphism $\Gamma(\cX)\rightarrow\Gamma(\cY)$ maps maximal faces onto maximal faces follows from Corollary~\ref{cor:maximalfaceclassify}.
\end{proof}

\begin{proof}[Proof of Theorem \ref{main3}]
  (i) $\Longrightarrow$ (ii).
A linear isometry  $A\colon\cX\to\FF^n$ is a diffeomorphism, so
it maps smooth points onto smooth points and induces a graph isomorphism between $\Gamma(\cX,\norm)$ and $\Gamma(\FF^n,\norm_{\infty})$. Therefore, it suffices to  prove that the equality in \eqref{89} of Corollary~\ref{inequality} is achieved  with the supremum norm $$\|(x_1,\dots,x_n)\|_{\infty}=\max\{|x_1|,\dots,|x_n|\}$$ on $\mathbb F^n$. Note that this norm is differentiable at a point $x = (x_1, \dots, x_n)$ if and only if exactly one of the components of $x$ has a maximal absolute value. Moreover, if exactly the $i$-th component of $x\in\FF^n$  has a maximal absolute value, then $\partial\|x\|_{\infty}=\{\mu e_i^\ast\}$  for some unimodular $\mu$, where $e_i$ is the  $i$-th standard basis vector and $e_i^*$ is a dual functional, i.e., $e_i^\ast(e_j)=\delta_{ij}$, the Kronecker delta.  Therefore, $$x^\bot=\Ker e_i^\ast $$ consists of all vectors with zero $i$-th component. As such, the outgoing neighborhood of a smooth projective point $[x]:=\FF x\in\Gamma(\cX)$ consists of all the lines in $\Ker e_i^\ast$ and thus in this case we obtain an equality in \eqref{89}.

(ii) $\Longrightarrow$ (iii). Without loss of generality, we may assume that $\cX=\mathbb F^n$ with $n=\dim\cX$. Now, the implication immediately follows from \eqref{89} and the fact that the equality can be achieved if the norm on $\cX$ is supremum norm.

(iii) $\Longrightarrow$ (i). If (iii) holds, then, by Lemma \ref{exposed-smooth}, there exist exactly $n$ different kernels corresponding to the exposed points in $\cX^\ast$. Therefore, if $x_1^\ast,\dots,x_n^\ast\in\cX^\ast$ are exposed points with pairwise different kernels, then every other exposed point in $\cX^\ast$ is a scalar multiple of one of them. As such, the set of exposed points in $\cX^\ast$ equals
   $$\TT x_1^\ast\cup\dots\cup\TT x_n^\ast$$ where $\TT=\{-1,1\}$ (if $\FF=\RR$) or $\TT$ is a unimodular group (if $\FF=\CC$). In both cases, the set of exposed points is compact and hence their convex span is already a closed subset. Therefore, in $\cX^\ast$ the extreme points coincide with exposed ones. By a suitably chosen linear bijection, we can assume that the extreme points in~$\cX^\ast$ are $\TT e_1^\ast,\dots,\TT e_n^\ast$. Clearly then, the norm on $\cX^\ast$ is a taxi-cab,   so~(i) holds by duality.
   \end{proof}

\section{BJ-norms and their properties} \label{section:BJ-norms}
Let us start by proving the following lemma, which  uses  the fact that $\Gamma$ is a quotient graph of $\Gamma_0\setminus\{0\}$. It
indicates  once again that $\Gamma$ is more restrictive than~$\Gamma_0$.
\begin{lem}\label{thm999}
Let $\cX$ and $\cY$ be normed spaces over $\FF$. If $\Gamma(\cX)$ is isomorphic to $\Gamma(\cY)$, then also $\Gamma_0(\cX)$ is isomorphic to $\Gamma_0(\cY)$.
\end{lem}
\begin{proof}
Let $\Phi\colon\Gamma(\cX)\rightarrow\Gamma(\cY)$ be  an isomorphism. For every projective point $\ell\in\PP\cX$ and $\Phi(\ell)\in\PP\cY$, choose their representatives $\ell'\in\ell$ and $\Phi(\ell)'\in\Phi(\ell)$, respectively. Then the induced map $\Psi\colon\Gamma_0(\cX)\to\Gamma_0(\cY)$, defined by
$$\Psi(\lambda \,\ell')=\lambda \Phi(\ell)',$$
is a  graph isomorphism.
\end{proof}
However, as we show at the end of this section, $\Gamma_0(\cX)$ and $\Gamma_0(\cY)$ being  isomorphic does not always imply that  $\Gamma(\cX)$ and $\Gamma(\cY)$ are isomorphic. We now introduce a class of norms, where the converse of the above lemma is true. Within the introduced class we could easily adapt Lemmas~\ref{lemma:two-dimensional-projective} and~\ref{lemma:high-dimensional-projective} to  the setup of  orthodigraph~$\Gamma_0(\cX)$ of the original normed space~$\cX$  rather than its projectivization,~$\Gamma(\cX)$.

\begin{defn}
We say that the elements $x, y \in \cX$ are {\em BJ-equivalent} (denoted by $x \simbj y$) if $x^\bot=y^\bot$ and ${}^\bot x={}^\bot y$.
\end{defn}
Since BJ-orthogonality is homogeneous, every two nonzero points in $\FF x$ are BJ-equivalent.  However, the converse is not always true: in $(\RR^3,\norm_{\infty})$ the points $(1,1/2,0)$ and $(1,1/3,0)$ are clearly linearly independent, but still are BJ-equivalent~(see~\cite[Example~3.9]{AGKRZ23}).
\begin{defn}\label{de:BJ-norm}
    A norm is called a \emph{BJ-norm} if BJ-equivalent elements are always linearly dependent.
\end{defn}

Observe that, in a BJ-normed space~$\cX$, the equivalence classes induced by BJ-equivalency relation are $\{0\}$ and points in $\PP\cX$.

\begin{lem}\label{smooth-rotund}
If a norm is rotund or smooth, then it is a  BJ-norm.
\end{lem}
\begin{proof}
   Assume first that $\cX$ has a rotund norm. Take any $x, y \in \cX \setminus \{ 0 \}$  with $x^\bot=y^\bot$. From~\eqref{eq15}, it is straightforward that  this is equivalent to $[x]^\bot=[y]^\bot$, which,   by \cite[Lemma~2.6]{AGKRZ23}, implies $[x] = [y]$, as claimed.

   Assume now that $(\cX,\norm)$  is a smooth normed space. Take two linearly independent $x,y\in\cX$ and choose a norm-attaining functional $f$  with $f(x)\neq0$ and $f(y)=0$ (it exists by Hahn-Banach theorem). Without loss of generality we may assume that $\|f\|=1$.
   Choose $z\in\cX$ such that $f$ is a supporting functional at $z$. Then $z\in {}^\bot y$. But since $\norm$ is smooth norm, $f$ is the unique supporting functional at $z$ and $f(x)\neq 0$, so $z\notin {}^\bot x$. Hence, ${}^\bot x\neq{}^\bot y$.
\end{proof}

As we will show in the last section, the class of BJ-norms is broader and contains not only smooth or strictly convex norms. However, as we have already mentioned, not every norm is a BJ-norm.

For the class of BJ-normed spaces $\cX$, we can easily see that smooth points can be described not only in terms of the projective orthodigraph $\Gamma(\cX)$ but also from the digraph $\Gamma_0(\cX)$ whose vertex set consists of all vectors in~$\cX$.

Recall that $0$ is the only vertex in $\Gamma_0(\cX)$ with a loop.
\begin{cor} \label{corollary:two-dimensional}
  Let $\cX$ be a two-dimensional BJ-normed space over $\FF$. Then the following conditions are equivalent for $x\in\Gamma_0(\cX)$:
  \begin{itemize}
    \item[(i)] $x$  corresponds to  a smooth point in $\cX$.
    \item[(ii)]  $x\in\Gamma_0(\cX)$ has no loops and every two elements of $x^\bot$ are BJ-equivalent.
  \end{itemize}
\end{cor}
\begin{proof}
Immediately follows from Lemma~\ref{lemma:two-dimensional-projective} and Definition~\ref{de:BJ-norm} of a BJ-norm.
\end{proof}
We call a subset $\Omega\subseteq \Gamma_0(\cX)$ to be a \emph{BJ-set} if any two vertices in $\Omega$ are BJ-equivalent. In other words, $\Omega$ is completely contained in a single BJ-equivalence class.
  \begin{cor} \label{corollary:high-dimensional}
Let $\cX$ be a finite-dimensional BJ-normed space over $\FF$ with $n = \dim\cX\ge 3$. Then the following conditions are equivalent for  $x\in\Gamma_0(\cX)$:
\begin{itemize}
  \item[(i)] $x$  corresponds to  a smooth point in $\cX$.
  \item[(ii)] $x\in\Gamma_0(\cX)$ has no loop and there exist  $(n-2)$ loopless vertices $x_3,\dots,x_{n}$ in $\Gamma_0(\cX)$ such that
  $x^\bot\cap x_3^\bot\cap\dots\cap x_{n}^\bot$ is a BJ-set.
  Moreover, for every reordering of this $(n-2)$-tuple of vertices we can define  $$\Omega:=\{y\in\Gamma_0(\cX);\;\;
  x^\bot\cap y^\bot\cap x_4^\bot \cap \dots\cap x_{n}^\bot \hbox{ is a  BJ-set}\},$$ and then $x_4^\bot\cap\dots\cap x_n^\bot\cap x^\bot\cap\bigcap\limits_{y\in\Omega}y^\bot=\{0\}$.
\end{itemize}
\end{cor}

\begin{proof}
Follows from Lemma~\ref{lemma:high-dimensional-projective} and Definition~\ref{de:BJ-norm} of a BJ-norm.
\end{proof}

The following statement is immediate from the above two corollaries: if $\cX,\cY$ are BJ-normed spaces over $\FF$ with $\dim\cX<\infty$, then every digraph isomorphism $\Phi\colon\Gamma_0(\cX)\to \Gamma_0(\cY)$ maps smooth points onto smooth points.

Next, we prove that the converse of Lemma \ref{thm999} holds for BJ-normed spaces.

\begin{lem} \label{lemma:isomorphism-implication}
Let $\cX$ and $\cY$ be BJ-normed spaces over $\mathbb F$. If $\Gamma_0(\cX)$ and $\Gamma_0(\cY)$ are isomorphic, then $\Gamma(\cX)$ and $\Gamma(\cY)$ are also isomorphic.
\end{lem}
\begin{proof}  Let $\Phi\colon\Gamma_0(\cX)\rightarrow\Gamma_0(\cY)$ be an isomorphism. For every projective point $\ell\in\PP\cX$ choose its representative $\ell'\in\ell$. Since $0\in\Gamma_0(\cX)$ is the only vertex with a loop, we clearly have $\Phi^{-1}(0)=\{0\}$. Therefore, $$\Psi(\ell) := [\Phi(\ell')]$$
is a well-defined map from $\Gamma(\cX)$ to $\Gamma(\cY)$. Clearly, $\Psi$ is a strong graph homomorphism. Since $\cX$ is a BJ-normed space, for all $x\neq 0$ it holds that $$\FF x=\{y\in\cX; \; y^\bot=x^\bot\text{ and } {}^\bot y={}^\bot x\}.$$ Since $\cY$ is also a BJ-normed space and $\Phi$ is an isomorphism, we have $$\Phi(\FF x) =\FF\Phi(x).$$
Then injectivity and surjectivity of the induced map $\Psi$ are straightforward.
\end{proof}

Furthermore, in Examples~\ref{example:nonsmooth-Radon-plane} and~\ref{example:complex-Radon-plane} we construct a nonsmooth two-dimensional normed space $\cX$ over $\FF \in \{ \RR, \CC \}$ such that there exists a (nonhomogeneous) bijection between $\cX$ and a Hilbert space $(\FF^2, \norm_2)$ which preserves BJ-orthogonality in both directions. In other words, $\Gamma_0(\cX)$ and $\Gamma_0(\FF^2, \norm_2)$ are isomorphic but, by Lemma~\ref{lemma:two-dimensional-projective}, $\Gamma(\cX)$ and $\Gamma(\FF^2, \norm_2)$ are not isomorphic. This will show that the  BJ-norm assumption  is essential in Corollary~\ref{corollary:two-dimensional}.

By~\cite[Theorem~1]{James0}, if $\dim \cX \geq 3$, then BJ-orthogonality in $\cX$ is symmetric if and only if the norm on $\cX$ is induced by the inner product. However, this result does not hold for $\dim \cX=2$. Two-dimensional normed spaces, in which BJ-orthogonality is symmetric but the norm is not induced by an inner product, are called {\em Radon planes}. Their first examples are due to Birkhoff~\cite{Birkhoff} and James~\cite[p.~561]{James0}. We provide a complete characterization of {\em real} Radon planes which was obtained by Day~\cite{Day}. We denote $x \pperp y$ if $x \perp y$ and $y \perp x$ for some $x, y \in \cX$.

\begin{rem}[{\cite[pp. 330--333]{Day}}] \label{remark:Radon-planes}
Let $\cX$ be a two-dimensional real normed space. Then BJ-orthogonality in $\cX$ is symmetric if and only if modulo a linear transformation its unit sphere ${\bf S}_{\cX}$ can be obtained by the following procedure:
\begin{itemize}
    \item[(i)] Choose  (auxiliary) any two-dimensional real normed space~$\cY$.
    \item[(ii)] Find any two normalized vectors $x,y \in \cY$ with $x \pperp y$ (they always exist by~\cite[Theorem~2]{Taylor}).
    \item[(iii)] Find supporting functionals $f_x, f_y \in \cY^*$ with $f_x(x) = f_y(y) = 1$ and $f_x(y) = f_y(x) = 0$.
    \item[(iv)] Choose a coordinate system in $\cY$ with $x$ and $y$ at $(1,0)$ and $(0,1)$, correspondingly. Similarly, choose a coordinate system in $\cY^*$ with $f_x$ and $f_y$ at $(0,1)$ and $(-1,0)$.
    \item[(v)] Set the first and the third quadrants of ${\bf S}_{\cX}$ equal to the first and the third quadrants of ${\bf S}_{\cY}$, set the second and the fourth quadrants of ${\bf S}_{\cX}$ equal to the second and the fourth quadrants of ${\bf S}_{\cY^*}$.
\end{itemize}
\end{rem}

Tanaka~ showed in~\cite[Theorem 4.7]{Tanaka22}  that orthodigraphs of real smooth Radon planes are  isomorphic to the orthodigraph of the Euclidean plane. We now extend his result. We will rely on the following simple lemma.

\begin{lem}\label{lem:countableperturbation}
    Let $f\colon[a,b]\to\RR$ be a nonconstant continuous function and let $g\colon[a,b] \to \RR$ be a function whose image is  at most countable. Then the image of $f+g$  contains continuum many points.
\end{lem}

\begin{proof}
The image of the continuous  function $f$ is a nondegenerate interval.  Let $(x_t)_t\in[a,b]$ be continuum many points    which $f$ maps bijectively onto $\Xi :=(f(x_t))_t$. Let $\Omega=g([a,b])$ be at most countable set. Then the same is true for
$$\Omega':=\Omega-\Omega=\bigcup_{y\in\Omega} (\Omega-y),$$
and the set $\Xi$ is partitioned into continuum many equivalence classes under the equivalence relation $f(x_t)\simeq f(x_s)$ if $f(x_t)-f(x_s)\in\Omega'$. The $f$-preimages of these equivalence classes partition the set $(x_t)_t$ into continuum many disjoint sets. And $f+g$ is injective on the representatives because if $(f+g)(x'_t)=(f+g)(x
'_s)$ for two representatives of distinct preimages, then $f(x'_t)-f(x'_s)= g(x'_s)-g(x'_t)\in\Omega-\Omega=\Omega' $, which contradicts the fact that $f(x'_t)$ and $f(x'_s)$ are not equivalent.
\end{proof}

For brevity, we will refer to the real affine dimension of a norm's face~$F$ as simply \emph{the dimension of $F$.}

\begin{lem} \label{lem:common-part-Radon-planes}
Let ${\mathcal R}:=(\FF^2,\norm_{\mathcal{R}})$ be a Radon plane over $\FF \in \{ \RR, \CC \}$, and $x \in {\mathcal R}$ be a point in the norm's unit sphere. Then the following statements hold.
  \begin{itemize}
    \item[(i)] There exists a norm's face~$F_x$ such that
    \begin{equation}\label{eq:cmplxradon3}
    x^\bot=\FF F_x.
    \end{equation}
    The dimension of $F_x$ equals $0$ if and only if $x$ is a smooth point.
    \item[(ii)] If, moreover, the boundary points of each nonzero-dimensional face in~${\mathcal R}$ are smooth, then:
    \begin{itemize}
    \item[(a)] If $x$ is nonsmooth, then $z^\bot=\FF x$ for any $z \in F_x$. Consequently, each point in $F_x$ is BJ-orthogonal only to $\FF x$, and $x$ is BJ-orthogonal only to $\FF F_x$.
    \item[(b)] Given an arbitrary nonzero-dimensional norm's face $F$ in~${\mathcal R}$, we have $\FF F \subseteq \FF F_x$ for some nonsmooth normalized point $x \in {\mathcal R}$.
    \end{itemize}
  \end{itemize}
\end{lem}

\begin{proof}
(i). Assume first that $x$ is nonsmooth. Then it has at least two $\FF$-linear supporting functionals, and their kernels are distinct one-dimensional subspaces of $\FF^2$.   So, by \cite[Theorem 2.1]{James1}, there exist at least two normalized linearly independent vectors
   $$y_1, y_2\in x^\bot,$$ and since $\norm_{\mathcal R}$ is a Radon norm, we also have $x\in y_1^\bot\cap y_2^\bot$.
   Then, again by~\cite[Theorem 2.1]{James1},
   there exist supporting functionals $f_1,f_2\colon\FF^2\to\FF$ of $y_1$ and $y_2$, respectively, such that
\begin{equation}\label{eq:cmplxradon}
f_1(x)=f_2(x)=0.
\end{equation}
Then their kernels are the same, and hence $f_1, f_2$ must be linearly dependent. Since they are normalized, $f_2=\mu f_1$ for some unimodular number $\mu$. By replacing $y_2 $ with $\mu y_2\in x^\bot$,  we achieve that $y_1,(\mu y_2)$ both share the same supporting functional $f:=f_1=f_2$. Following along the same lines shows that for  any normalized $z\in x^\bot$ there exists a unimodular number $\mu_z$ such that $f$ is a supporting functional of $\mu_z z$.
It is straightforward that then there exists a convex set
\begin{equation} \label{eq:exposed-face}
F_x :=f^{-1}(1)\cap {\bf B}_{\mathcal R}
\end{equation}
(where ${\bf B}_{\mathcal R}$ is the norm's closed unit ball) such that $x^\bot\subseteq\FF F_x$.
Conversely, if $z\in F_x$, then $f$ is its supporting functional which, by \eqref{eq:cmplxradon},  annihilates $x$, so $z\perp x$, and since $\mathcal{R}$ is a Radon plane, this implies $x\perp z$, that is, $z\in x^\bot$. Thus $x^\bot=\FF F_x$. Note that, by \eqref{eq:exposed-face} and Lemma~\ref{proposition:exposed-faces}, the set $F_x$ is actually a norm's (exposed) face.

Clearly, \eqref{eq:cmplxradon3} holds also for a smooth point $x$, if we let $F_x$ to be the norm-one representative in the one-dimensional kernel of the supporting functional at~$x$.

\medskip

(ii). We first prove item (b). It is not hard  to see that each point~$z$ in the relative interior of $F$ is smooth: namely, $z$ is  the middle point of some nontrivial line segment $$[y_1,y_2]\subseteq F,$$
and then, if $f$ is any $\FF$-linear supporting functional for $z=\frac{y_1+y_2}{2}$, it follows that $$1=f(z)=\frac{f(y_1)+f(y_2)}{2}.$$
Since $\|f\|=\|y_1\|=\|y_2\|=1$, we have $|f(y_1)|, |f(y_2)| \leq 1$, and then the triangle inequality implies that
\begin{equation}\label{eq:cmplx-rad2}
f(y_1)=f(y_2)=1.
\end{equation}
Notice also that $y_1$ and $y_2$ must be $\FF$-linearly independent because otherwise $y_2=\mu y_1$ for some unimodular $\mu \neq 1$, and then $\| z \| = \left\|\frac{y_1+y_2}{2}\right\|=\frac{|1+\mu|}{2} < 1$, a contradiction. Therefore, they form a basis for $\FF^2$. It then follows from \eqref{eq:cmplx-rad2} that the supporting functional at $z$  is unique, so $z$ is smooth.
By the hypothesis, the boundary points of $F$ are also smooth. Then,  by (ii) $\Longrightarrow$ (i) of Lemma~\ref{thm5} (whose proof can be easily adapted to work also in nonprojective setting), for every $z,w\in F$ we have $z^\bot=w^\bot$, and it is a one-dimensional subspace of~$\FF^2$ (spanned by, say, $x$), so
\begin{equation}\label{eq:cmplxradon4}
    \bigcup_{z\in F} z^\bot=\FF x.
\end{equation}
Since $\norm_{\mathcal R}$ is a Radon norm, together with \eqref{eq:cmplxradon3} this gives us that $\FF F \subseteq \FF F_x$.

To prove (a), note that, for a nonsmooth point~$x$, the dimension of~$F_x$ is greater than~$0$, so we can apply~\eqref{eq:cmplxradon4} to $F = F_x$.
\end{proof}

\begin{lem}\label{lem:nnnsmoothRadonplaneRR}
    Let ${\mathcal R}:=(\RR^2,\norm_{\mathcal{R}})$ be a real Radon plane and let $\|(x,y)\|_2:=\sqrt{|x|^2+|y|^2}$ be the Euclidean norm.  Then the following conditions are equivalent:
\begin{itemize}
\item[(i)] $\Gamma_0({\mathcal R})$ is isomorphic to $\Gamma_0(\RR^2,\norm_2)$.
\item[(ii)] Both   boundary points of each one-dimensional face in ${\mathcal R}$ are smooth.
\end{itemize}
\end{lem}
\begin{proof}
(ii) $\Longrightarrow$ (i).
  If $\norm_{\mathcal R}$ is smooth, then this follows by Tanaka's~\cite[Theorem~4.7]{Tanaka22}. Otherwise, by Lemma~\ref{lem:common-part-Radon-planes}, for each nonsmooth point $x$ there exists a one-dimensional norm's face $F_x$ such that each point in $F_x$ is BJ-orthogonal only to $\RR x$, and $x$ is BJ-orthogonal only to $\RR F_x$. Conversely, each point of a one-dimensional norm's face $F$ is BJ-orthogonal to a unique one-dimensional subspace spanned by a nonsmooth point. There  are at most countably many nonsmooth points, and hence also the same cardinality of maximal one-dimensional faces; they can be paired into  mutually BJ-orthogonal pairs $(x,F_x)$, so that no other nonzero point in $\mathcal{R}$ is in BJ-orthogonality relation to any nonzero point in $\RR x\cup \RR F_x$.

  We next show that there are continuum many smooth points in the norm's unit sphere, which do not belong to a one-dimensional face. To see this,  note that the  norm's unit sphere  is not a square, so its portion can be parametrized as $(x,f(x))$
for some nonconstant  convex function
$f\colon(\alpha,\beta)\to\RR$.
 It is a classical result~\cite[Theorems~23.1 and~24.1]{Rockafellar} that $f(x)$ has left, $f'_-(x)$, and right, $f'_+(x)$, derivative at every  point of its domain,  they monotonically increase and coincide except possibly at countably many points. The graph of its subdifferential,
$${\mathcal G}_{\partial f}:=
\{(x, x^\ast)\in \RR^2;\; x^\ast\in\RR,\; f'_-(x)\le  x^\ast \le f'_+(x)\},$$
is a continuous curve in $\RR^2$ which may contain horizontal and vertical line segments. Moreover,
 $$(x,x^\ast)\in{\mathcal G}_{\partial f} \quad\hbox{ if and only if }\quad(x,f(x))\perp (1,x^\ast).$$
Indeed, if $(x, f(x))$ is a smooth point, then its supporting functional is given by norm's gradient at this point, which is perpendicular to the tangent vector of the level curve parametrized by $x\mapsto (x,f(x))$. If $(x, f(x))$ is a nonsmooth point, then, by using~\cite[Theorem 25.6]{Rockafellar}, we obtain that its subdifferential is the convex hull of all possible limits of gradients at smooth points on ${\bf S}_{\mathcal R}$ converging to $(x, f(x))$.
It follows that each  vertical segment in ${\mathcal G}_{\partial f}$ corresponds  to a point in $(x,f(x))^\bot$ for a fixed nonsmooth point
$x\in(\alpha,\beta)$, and it also follows that   two points $(x_1,x_1^\ast),(x_2,x_2^\ast)\in{\mathcal G}_{\partial f}$ do not belong to the same horizontal/vertical line segment if and only if $(x_1,f(x_1))^\bot\neq(x_2,f(x_2))^\bot$.

We may clearly assume that  $f'$ is defined also at points where $f$ is not differentiable, by requiring  that $x\mapsto f'(x)$ is right continuous. Since it  is increasing, a quick application of Stieltjes integral  shows that there exists a finite positive measure $\mu$, supported on the  interval $(a,b)$, such that  $f'(x)=\mu((a,x])$ is its cumulative distribution function. Consider a purely atomic measure $\mu_p$, supported on the at most countable set of nonsmooth points
 $x_1,x_2,\dots\in(a,b)$ and defined by $$\mu_s(\{x_i\}):=f'_+(x_i)-f'_-(x_i),$$ i.e., the length  of the vertical line segment of ${\mathcal G}_{\partial f}$  at $X_i$.

 Notice that $\mu-\mu_p$ is a continuous measure. It cannot be zero, because then $f'$ would be a cumulative distribution function of a purely atomic measure $\mu_{p}$, and as such ${\mathcal G}_{\partial f}$ would consist only of horizontal and vertical line segments, so that the end points of one-dimensional norm's faces (which correspond to horizontal line segments of ${\mathcal G}_{\partial f}$) would be nonsmooth, a contradiction. Then $f'$ is a sum of a nonconstant continuous increasing function $g\colon x\mapsto(\mu-\mu_p)((a,x]
 )$ and an increasing  function which is constant except for jumps at $x_1,x_2,\dots$. Thus, by Lemma~\ref{lem:countableperturbation}, $f'(x)$ attains continuum many different values. By the above, this translates to the fact that the norm's unit sphere contain continuum many smooth points $x$ with $x^\bot$ pairwise distinct.

 We can now construct  the graph isomorphism  $\Phi\colon\Gamma_0({\mathcal R})\to \Gamma_0(\RR^2,\norm_2)$ as follows:  Let ${\mathcal SF}$ be a collection of  mutual BJ-orthogonal pairs $(\RR {x_i},\RR F_{x_i})$, where $x_i$ is a nonsmooth point on the norm's unit sphere, and $F_{x_i}$ is the corresponding face. It contains at most countably many pairs. With each such pair we associate a pair of orthogonal points $((c_i,s_i), (-s_i,c_i))\in(\RR^2,\norm_2)$. By the last paragraph, $(\RR^2,\norm_{\mathcal{R}})$  still contains continuum many mutual BJ-orthogonal pairs $(\FF x, \FF x')$ consisting of smooth points.  Then there exist  bijections $\varphi_i\colon \RR F_{x_i}\setminus\{0\}\to\RR\setminus\{0\}$, and also a bijection $\varphi\colon \RR^2\setminus \bigcup_i(\RR x_i\cup \RR F_{x_i})$ onto $\RR^2\setminus\bigcup_i(\RR(c_i,s_i)\cup\RR(-s_i,c_i))$ which maps smooth mutual BJ-orthogonal pairs  onto orthogonal pairs. Hence the mapping $\Phi: (\RR^2,\norm_{\mathcal R}) \to (\RR^2,\norm_2)$ defined by
 \begin{itemize}
\item[(i)] $\Phi(0)=0$,
\item[(ii)] $\Phi(\lambda x_i)= \lambda (c_i,s_i)$ and $ \Phi(\lambda z)=\varphi_i(\lambda z) (-s_i,c_i)$  for $z\in F_{x_i}$ and $\lambda\in\RR\setminus\{0\}$,
\item[(iii)] $\Phi(x)=\varphi(x)$ if $x\notin \bigcup_i (\RR x_i\cup \RR F_{x_i})$
 \end{itemize}
 is a bijection which preserves BJ-orthogonality in both directions, i.e., it is a graph isomorphism between $\Gamma_0(\RR^2,\norm_{\mathcal R})$ and $\Gamma_0(\RR^2,\norm_2)$.
\bigskip

$\neg$(ii) $\Longrightarrow$ $\neg$(i). Suppose that $F=[x_1,x_2]$ is a one-dimensional norm's face with a nonsmooth end point $x_1$. Then $x_1^\bot$ properly contains $y^\bot$, where $y=\frac{x_1+x_2}{2}$ is the midpoint of $F$. Notice that such property cannot take place in Hilbert spaces, so there can be no graph isomorphism between $\Gamma_0(\RR^2,\norm_{\mathcal R})$ and $\Gamma_0(\RR^2,\norm_2)$.
\end{proof}

We have a similar classification for complex Radon planes:
\begin{lem}\label{lem:nnnsmoothRadonplaneCC}
    Suppose ${\mathcal R}:=(\CC^2,\norm_{\mathcal{R}})$ is a complex Radon plane and let $(\CC^2,\norm_2)$ be a Hilbert space with the norm $\|(x,y)\|_2:=\sqrt{|x|^2+|y|^2}$.  Then the following conditions are equivalent:
\begin{itemize}
\item[(i)] $\Gamma_0({\mathcal R})$ is isomorphic to $\Gamma_0(\CC^2,\norm_2)$.
\item[(ii)]  (a) The boundary points of each nonzero-dimensional face in ${\mathcal R}$ are smooth,~and\\ (b) the set $\{x^\bot; \;\;x\in\Gamma_0({\mathcal R})\}$ has continuum many points.
\end{itemize}
\end{lem}
\begin{proof}
(ii) $\Longrightarrow$ (i).
By Lemma~\ref{lem:common-part-Radon-planes}, for any normalized point~$x$ there exists a norm's face $F_x$ such that $x^\bot=\CC F_x$, and $F_x$ is zero-dimensional if and only if $x$ is a smooth point. Besides, if $x$ is nonsmooth, then each point in~$F_x$ is BJ-orthogonal only to $\CC x$, and $x$ is BJ-orthogonal only to the points in $\CC F_x$. Moreover, if $F$ is a norm's face of dimension greater than~$0$, then $\CC F \subseteq \CC F_x$ for some nonsmooth normalized point $x$. As a consequence,   the whole space~$\CC^2$ decomposes into a  union
\begin{equation}\label{eq:complex-unuion}
\CC^2=\{0\} \cup\bigcup_x (\CC^\ast x\cup\CC^\ast F_x);\qquad \CC^\ast=\CC\setminus\{0\}
\end{equation}
  where the union runs over all $\norm_{\mathcal{R}}$-normalized vectors $x$ which belong to  zero-dimensional maximal faces. Moreover,  the points in $\CC^\ast x$ and in $\CC^\ast F_x$ are mutually BJ-orthogonal, but are  BJ-orthogonal to no other nonzero point and no other nonzero point is BJ-orthogonal to a point in $\CC^\ast x\cup\CC^\ast F_x$. In addition, either $(\CC^\ast x\cup\CC^\ast F_x)$ and $(\CC^\ast y\cup\CC^\ast F_y)$ are equal or they are disjoint.

 Also, by part (b) of assumption (ii)
 there are continuum many such pairs.
 We can now construct  a graph isomorphism  $\Phi\colon\Gamma_0({\mathcal R})\to \Gamma_0(\CC^2,\norm_2)$ as follows. We partition the two-dimensional Hilbert space into continuum many pairwise disjoint pairs, consisting of mutually orthogonal vectors, by the formula
 \begin{equation}\label{eq:CC^2-euclid}
 \CC^2=\{0\}\cup\bigcup_{0\le t,\theta<\pi/2}\Bigl(\CC^\ast (\cos t,e^{4\theta i} \sin t)\cup\CC^\ast (-e^{-4\theta i} \sin t,\cos t)\Bigr).
  \end{equation}
  With each of the continuum many pairwise distinct  pairs $(\CC^\ast x,\CC^\ast F_x)$ that constitute the union in ~\eqref{eq:complex-unuion} we bijectively associate a pair, indexed by $t_x,\theta_x$, from the  union in \eqref{eq:CC^2-euclid}. For any $x$ there exists a bijection $\varphi_x\colon \CC^\ast F_x\to\CC^\ast$. Then the mapping $\Phi: (\CC^2,\norm_{\mathcal R}) \to (\CC^2,\norm_2)$ defined by
 \begin{itemize}
\item[(i)] $\Phi(0)=0$,
\item[(ii)] $\Phi(\lambda x)= \lambda (\cos t_x,e^{4\theta_x i}\sin t_x)$ and $ \Phi(\lambda z)=\varphi_x(\lambda z) (-e^{-4\theta_x i}\sin t_x,\cos t_x)$  for $z\in F_{x}$ and $\lambda\in\CC^\ast$
 \end{itemize}
 is a bijection which preserves BJ-orthogonality in both directions, i.e., it is a graph isomorphism between $\Gamma_0(\CC^2,\norm_{\mathcal R})$ and $\Gamma_0(\CC^2,\norm_2)$.
\bigskip

$\neg$(ii) $\Longrightarrow$ $\neg$(i).  Clearly, the orthogonality graph of the two-dimensional Hilbert space contains continuum many different outgoing neighborhoods. Namely, if $c,s$ are real numbers with $c > 0$ and $c^2+s^2=1$, then $(c,s)^\bot=\CC(-s,c)$ are pairwise distinct. In addition, every  isomorphism  $\Phi\colon\Gamma_0({\mathcal R})\to \Gamma_0(\CC^2,\norm_2)$ satisfies $\Phi(x^\bot)=\Phi(x)^\bot$. So, if part (b) of item (ii) does not hold for a complex Radon plane~${\mathcal R}$, there can be no isomorphism.
It remains to prove that no isomorphism exists in the case when there is a face~$F$ of~${\mathcal R}$ with $\dim_{\mathrm{aff}}F>0$ which contains a nonsmooth point on its relative boundary. Firstly, as shown before, the relative interior of $F$ consists of smooth points only.
Let $w\in F$ be a nonsmooth point on its boundary. Since  $\dim_{\mathrm{aff}}F>0$, there exists a  point $x\in F$ in its relative interior, so $x$ is smooth, and hence $x^\bot$ coincides with the one-dimensional kernel of the supporting functional at~$x$. One can easily adapt the proof of (ii) $\Longrightarrow$ (i) of Lemma~\ref{thm5} to work also in nonprojective setting, and it shows that  $x^\bot$ is properly contained in $w^\bot$. Notice that such property cannot take place in Hilbert spaces, so there can be no graph isomorphism between $\Gamma_0(\CC^2,\norm_{\mathcal R})$ and $\Gamma_0(\CC^2,\norm_2)$.
\end{proof}

An example of a  real Radon  plane $\mathcal R$ whose orthodigraph, $\Gamma_0(\mathcal R)$, has  only finitely many outgoing neighborhoods  is the  Hexagonal norm (it is the  image under a linear bijection of a norm which in the first quadrant coincides with $\norm_{\infty}$ and in the second quadrant coincides with its dual, $\norm_1$). Clearly, in this case $\Gamma_0(\mathcal R)$ cannot be isomorphic to $\Gamma_0(\RR^2,\norm_2)$, because the latter contains uncountably many different outgoing neighborhoods.

\medskip

We next show that there do exist nonsmooth real Radon planes $\mathcal{R}$ such that $\Gamma_0(\mathcal{R})$ and $\Gamma_0(\RR^2, \norm_2)$ are isomorphic.

\begin{ex} \label{example:nonsmooth-Radon-plane}
Consider an arbitrary normed space $(\RR^2, \norm)$ whose unit sphere satisfies the following conditions:
\begin{enumerate}
    \item[(i)] it is symmetric with respect both to the $x$-axis and the $y$-axis;
    \item[(ii)] it contains a nontrivial line segment parallel to the $x$-axis;
    \item[(iii)] it passes through $\pm (1,0)$ and $\pm(0,1)$;
    \item[(iv)] it is smooth at all points, except for, possibly, $\pm (1,0)$.
\end{enumerate}
An example of such a unit sphere is depicted in Figure~\ref{figure:original-norm}. Let us denote $x = (1,0)$ and $y = (0,1)$. Clearly, $x \pperp y$. Hence there exists a Radon plane $(\RR^2, \norm_{\mathcal R})$ such that the first and the third quadrants of its unit sphere coincide with those of $(\RR^2, \norm)$. In other words, only striped areas are changed in Figure~\ref{figure:original-norm}.

The $\perp$ relation is symmetric on $(\RR^2, \norm_{\mathcal R})$. Since the unit sphere is convex, for any $z$ from the horizontal line segment we have $z \perp x$, and thus $x \perp z$. Therefore, $x$ is a nonsmooth point of $(\RR^2, \norm_{\mathcal R})$ (even when it is a smooth point of $(\RR^2, \norm)$). We denote by $z_0$ the rightmost point of the horizontal line segment $I$ in the first quadrant of the unit sphere (if $x$ is a nonsmooth point of the original norm, then $I$ is also extended to the second quadrant).

Assume that there is a nonsmooth point $w \neq \pm x$ on the unit sphere of $(\RR^2, \norm_{\mathcal R})$. Then $w$ must belong to the second or to the fourth quadrant, and there exist linearly independent normalized $u, v$ such that $w \perp u$ and $w \perp v$. Since the unit sphere is convex and $x \perp z_0$, we may assume that $u$ and $v$ belong to the first quadrant and lie between $z_0$ and $x$. By the symmetry of $\perp$ on $(\RR^2, \norm_{\mathcal R})$, we have $u \perp w$ and $v \perp w$. But the norm's unit sphere is strictly convex on the interval between $z_0$ and $x$, so we obtain a contradiction.

Therefore, there are only two types of pairs of BJ-orthogonal lines in $(\RR^2, \norm_{\mathcal R})$:
\begin{enumerate}
    \item[(i)] $\RR x \pperp \RR z$ for any $z$ from the horizontal line segment $I$ of the unit sphere;
    \item[(ii)] all other lines are divided into pairs of mutually BJ-orthogonal lines.
\end{enumerate}

We denote $\RR^* I = \{ \alpha z; \;  \alpha \in \RR^*, \; z \in I \}$. Then it is easy to construct a (nonhomogeneous) bijective mapping $\Phi\colon (\RR^2, \norm_{\mathcal R}) \to (\RR^2, \norm_2)$ which preserves BJ-orthogonality in both directions:
\begin{enumerate}
    \item[(i)] $\Phi(0) = 0$;
    \item[(ii)] $\Phi(\RR^* x) = \RR^*(1,0)$;
    \item[(iii)] $\Phi(\RR^* I) = \RR^*(0,1)$;
    \item[(iv)] $\Phi$  maps the rest of  mutually BJ-orthogonal lines in $(\RR^2, \norm_{\mathcal R})$ bijectively onto the rest of mutually BJ-orthogonal lines in $(\RR^2, \norm_2)$.
\end{enumerate}
Therefore, $\Phi$ is a graph isomorphism between $\Gamma_0(\RR^2,\norm_{\mathcal R})$ and $\Gamma_0(\RR^2,\norm_2)$.
\end{ex}

\noindent\hspace{-0.05\linewidth}
\begin{minipage}{0.54\linewidth}
\begin{figure}[H]
\centering
\includegraphics[width=0.85\linewidth]{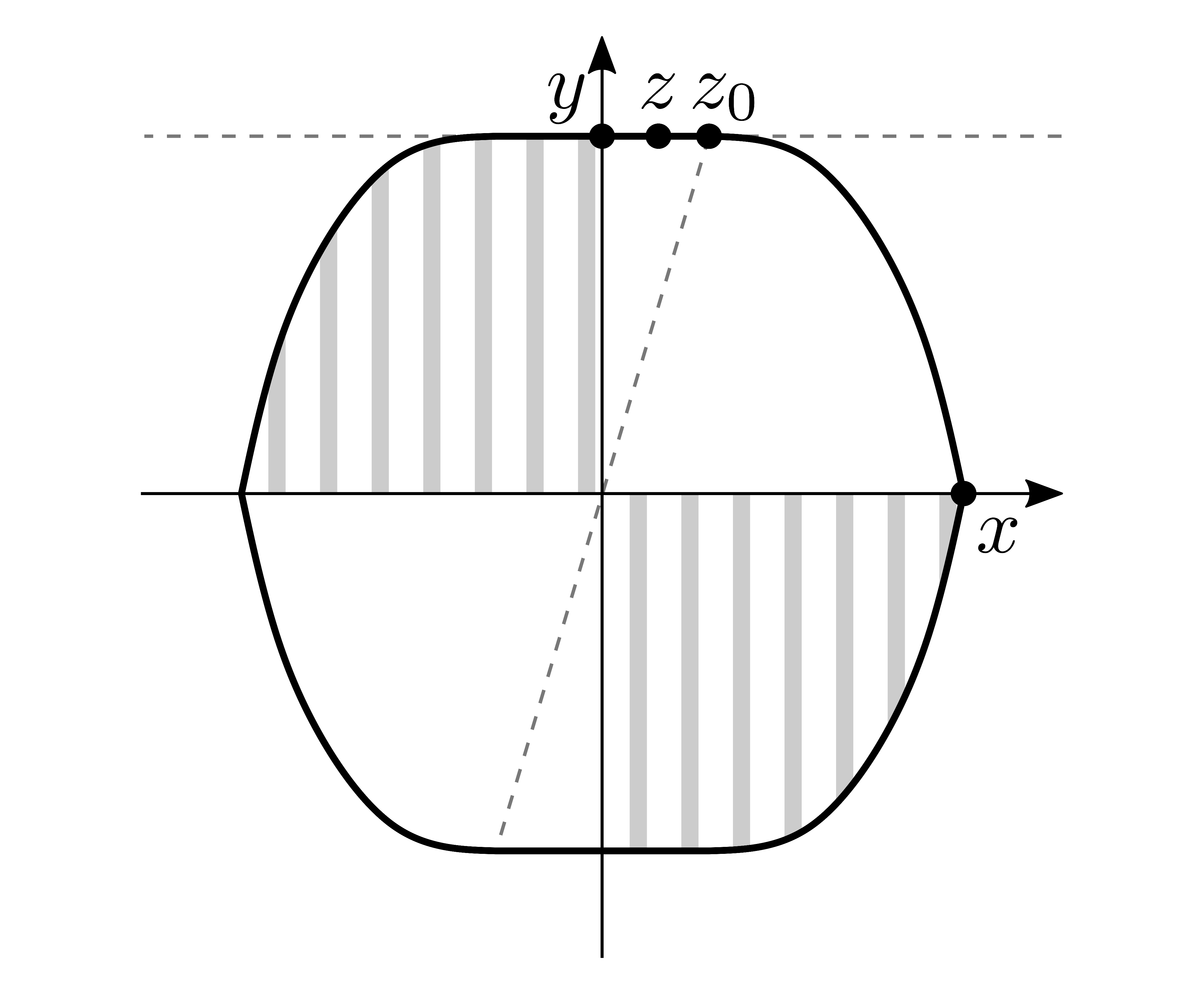}
\caption{\label{figure:original-norm} The original norm on~$\RR^2$.}
\end{figure}
\end{minipage}
\begin{minipage}{0.54\textwidth}
\begin{figure}[H]
\centering
\includegraphics[width=0.85\linewidth]{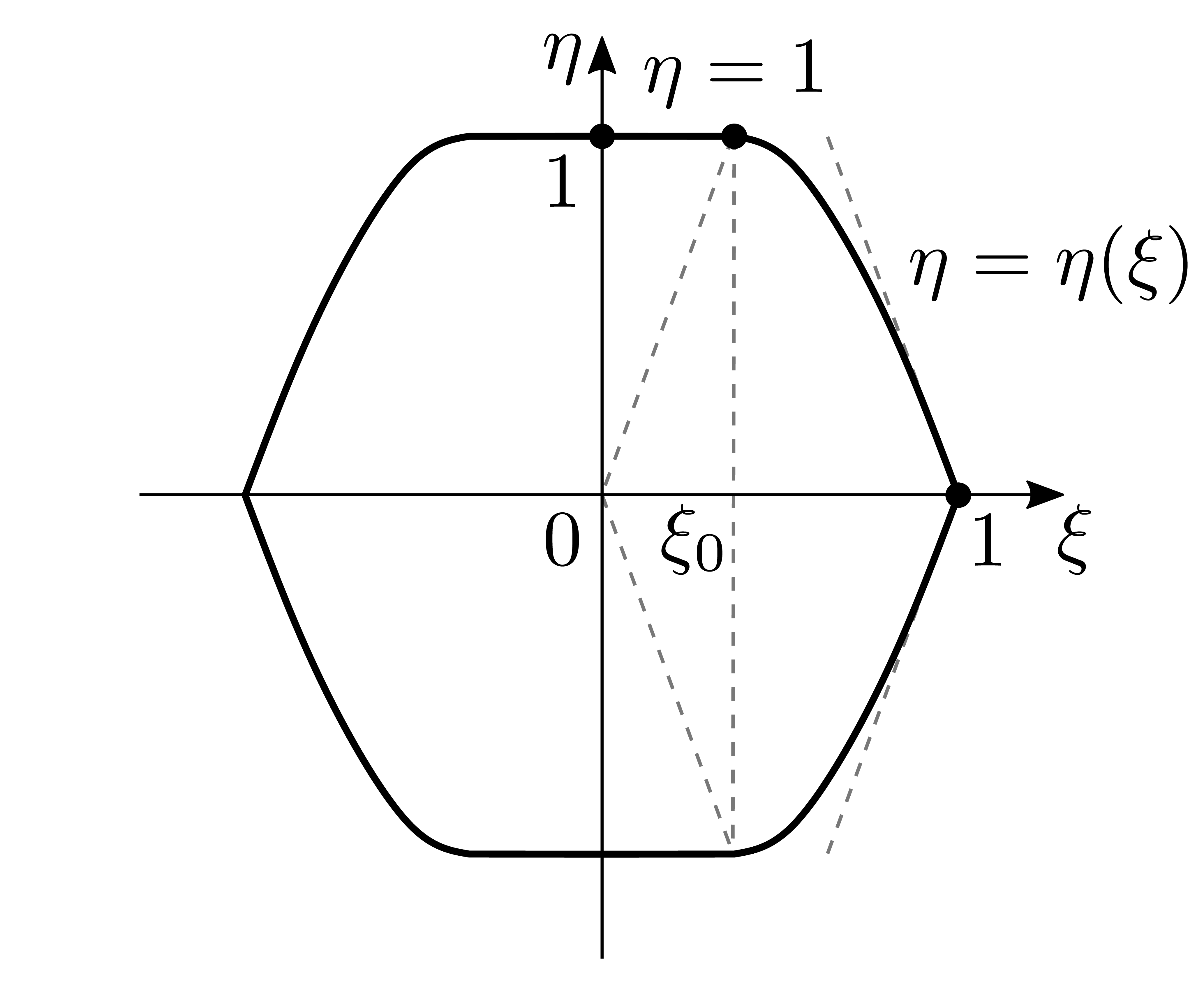}
\caption{\label{figure:nonsmooth-Radon} Absolute nonsmooth Radon plane.}
\end{figure}
\end{minipage}
\smallskip

For our further purposes, we will need an {\em absolute} nonsmooth real Radon plane~$\mathcal{R}$, i.e., a Radon plane such that its unit ball is nonsmooth but symmetric with respect both to the $x$-axis and the $y$-axis. The latter condition is equivalent to the fact that the norm on $\mathcal{R}$ satisfies the equality
$$
\| (a,b) \|_{\mathcal R} = \| (|a|, |b|) \|_{\mathcal R};\qquad (a,b) \in \RR^2.
$$
Note that the following example is a particular case of Example~\ref{example:nonsmooth-Radon-plane}.

\begin{ex} \label{example:absolute-Radon-plane}
Let $\cX$ be an absolute nonsmooth two-dimensional real normed space whose unit sphere is given by Figure~\ref{figure:nonsmooth-Radon} for some unknown decreasing smooth function $\eta(\xi)$ and unknown value $0 < \xi_0 < 1$. We now write down a system of equations which guarantees that $\cX$ is a Radon plane and present one of its solutions.

Consider  any  vector on the unit sphere $v = (\xi, \eta(\xi))$, $\xi_0 < \xi < 1$.  Recall that the supporting functional at $v$ equals the norm's gradient at~$v$ which is always perpendicular to the tangent vector of the norm's unit sphere. Therefore,  $v^\bot = \RR (1, \eta'(\xi))$, so $v$ is BJ-orthogonal to a unit vector $u = (\alpha(\xi), \alpha(\xi) \eta'(\xi))$, $\xi_0 < \alpha(\xi) < 1$, which belongs to the fourth quadrant. Since BJ-orthogonality on any Radon plane is symmetric and $u$ is a smooth point of the norm, we have $u^\bot = \RR v$, so the symmetry of our Radon plane with respect to the $x$-axis implies that $\alpha(\alpha(\xi)) = \xi$. By the continuity, $\alpha$ satisfies the boundary conditions $\alpha(\xi_0) = 1$ and $\alpha(1) = \xi_0$. For instance, we can take $\alpha(\xi) = \xi_0/\xi$.

Since $u$ is a unit vector from the fourth quadrant, we have
\begin{equation} \label{equation:nosmooth-1}
    \eta(\alpha(\xi)) = -\alpha(\xi) \eta'(\xi).
\end{equation}
Then $\alpha(\alpha(\xi)) = \xi$ implies that
\begin{equation} \label{equation:nosmooth-2}
    \eta(\xi) = -\xi \eta'(\alpha(\xi)).
\end{equation}
We now differentiate \eqref{equation:nosmooth-1}, multiply it by $\xi$ and substitute \eqref{equation:nosmooth-2} to obtain
\begin{equation} \label{equation:nosmooth-3}
    \xi \alpha(\xi) \eta''(\xi) + \alpha'(\xi) (\xi \eta'(\xi) - \eta(\xi)) = 0.
\end{equation}
By substituting $\alpha(\xi) = \xi_0/\xi$ and multiplying the resulting equation by $\xi^2 / \xi_0$, we get
\begin{equation} \label{equation:nosmooth-4}
    \xi^2 \eta''(\xi) - \xi \eta'(\xi) + \eta(\xi) = 0.
\end{equation}
Its general solution is $\eta(\xi) = c_1 \xi + c_2 \xi \ln \xi$, $c_1, c_2 \in \RR$. But $\eta(1) = 0$, so $c_1 = 0$. Since $\eta'(\xi_0) = 0$, we then have $\xi_0 = 1/e$, and $\eta(\xi_0) = 1$ implies that $c_2 = -e$. Thus $\eta(\xi) = -e \, \xi \ln \xi$. One can verify that $\eta(\xi)$ indeed satisfies \eqref{equation:nosmooth-1}.
\end{ex}

Example~\ref{example:absolute-Radon-plane} allows us to construct a nonsmooth {\em complex} Radon plane. The first examples of complex Radon planes were given by Oman~\cite[pp.~43--48, Constructions~III,~IV and~V]{Oman}, though he did not obtain a complete characterization for them. His constructions use the following result:

\begin{lem}[{\cite[Theorem~3.4]{Oman}}] \label{lemma:Radon-planes-criterion}
A two-dimensional normed space $\cX$ over $\FF \in \{ \RR, \CC \}$ is a Radon plane if and only if there exists an isometry $\Phi\colon\cX\to\cX^\ast$ such that $\Phi(x)(x)=0$, $x\in\cX$.
\end{lem}

\begin{lem}[{\cite[Theorems~3.4 and~3.5]{Oman}}] \label{lemma:orthogonality-criterion}
Let ${\mathcal R}=(\FF^2,\norm_{\mathcal R})$ be a Radon plane over $\FF \in \{ \RR, \CC \}$, $x, y \in \mathcal R$, $x \perp y$ and $\| x \|_{\mathcal R} = \| y \|_{\mathcal R} = 1$. Then for any $a,b,c,d \in \FF$ we have $|ad - bc| \leq \| ax + by \|_{\mathcal R} \cdot \| cx + dy \|_{\mathcal R}$, and the inequality becomes an equality if and only if $ax + by \perp cx + dy$.
\end{lem}

\begin{ex} \label{example:complex-Radon-plane}
Let $(\RR^2, \norm_{\mathcal R})$ be an absolute nonsmooth Radon plane from Example~\ref{example:absolute-Radon-plane}, and consider $\cX = \CC^2$ with the norm defined by
$$
\| (a,b) \|_{\cX} := \| (|a|, |b|) \|_{\mathcal R}.
$$
By~\cite[pp.~45--48, Constructions~IV and~V]{Oman}, this is indeed a normed space. Any linear functional in $\cX^*$ can be identified with $\widehat{(c, d)}$ for a certain $(c, d) \in \CC^2$ whose action is given by $\widehat{(c, d)}(a, b) = ca + db$. Then $\Ker \widehat{(-d, c)} = \CC (c,d)$ and $\| \widehat{(-d, c)} \|_{\cX^*} = \| (c,d) \|_{\cX}$, since for any $(a,b) \in \cX$ we have
\begin{align*}
|\widehat{(-d, c)}(a, b)| &= |-da + cb| \leq |da| + |cb|\\
&= |d| |a| + |c| |b| \leq \| (|c|, |d|) \|_{\mathcal R} \cdot \| (|a|, -|b|) \|_{\mathcal R}\\
&= \| (|c|, |d|) \|_{\mathcal R} \cdot \| (|a|, |b|) \|_{\mathcal R} = \| (c,d) \|_{\cX} \cdot \| (a,b) \|_{\cX}.
\end{align*}
The inequality in the second line follows from Lemma~\ref{lemma:orthogonality-criterion}. The norm of $\widehat{(-d, c)}$ is achieved on $(a, b) \in \cX$ if and only if $(|a|, -|b|) \perp (|c|, |d|)$ in $(\RR^2, \norm_{\mathcal R})$ and $\Arg (-da) = \Arg (cb)$, where $\Arg(\xi)$ denotes the argument of $\xi \in \CC$. We also assume here that $\Arg(0) = \Arg(\xi)$ for any $\xi \in \CC$.

Hence the mapping $\Phi\colon \cX \to \cX^*$ defined by $\Phi((c,d)) = \widehat{(-d, c)}$ is an isometry which satisfies Lemma~\ref{lemma:Radon-planes-criterion}, so $\cX$ is a complex Radon plane. Besides, $(a,b) \perp (c,d)$ in $\cX$ if and only if $(|a|, -|b|) \perp (|c|, |d|)$ and $\Arg (-da) = \Arg (cb)$.

Let us denote $x = (1,0)$ and $I = \{ (\xi, 1); \; \xi \in \CC, \; |\xi| \leq 1/e \}$. There are again only two types of pairs of BJ-orthogonal lines in $\cX$:
\begin{enumerate}
    \item[(i)] $\CC x \pperp \CC z$ for any $z \in I$;
    \item[(ii)] all other lines are divided into pairs of mutually BJ-orthogonal lines.
\end{enumerate}
Clearly, this means that $x$ is a nonsmooth point of $\cX$, so $\cX$ is a nonsmooth complex Radon plane. Similarly to Example~\ref{example:nonsmooth-Radon-plane}, we can construct a (nonhomogeneous) bijective mapping $\Phi\colon \cX \to (\CC^2, \norm_2)$ which preserves BJ-orthogonality in both directions, and this mapping is an isomorphism between $\Gamma_0(\cX)$ and $\Gamma_0(\CC^2, \norm_2)$.
\end{ex}

By using Examples~\ref{example:nonsmooth-Radon-plane} and~\ref{example:complex-Radon-plane} we can also show that the assumption that~$\cX$ is a BJ-normed space is essential in Corollary~\ref{corollary:high-dimensional}.
Let $\cX$ be a complex normed space, $x, y \in \cX$. We denote $x \perp_\RR y$ if $x$ and $y$ are BJ-orthogonal as elements of the real normed space~$\cX_\RR$, which is obtained from $\cX$ by restricting the scalars to the field of real numbers.

\begin{lem} \label{lemma:direct-sum}
Let $\cX$ and $\cY$ be normed spaces over $\FF$, and $\cZ = \cX \oplus_2 \cY$ with $\| (x, y) \|_{\cZ} = \sqrt{\| x \|_{\cX}^2 + \| y \|_{\cY}^2}$. Then for any nonzero $x \in \cX$ we have
\begin{align*}
    (x,0)^\bot &= x^\bot \oplus \cY,\\
    {}^\bot(x,0) &= {}^\bot x \oplus \cY.
\end{align*}
\end{lem}

\begin{proof}
Let us denote the norm's directional derivatives at $u$ in direction of $v$ by
$$
D_-(u;v):=\lim_{t\nearrow 0}\frac{\|u+tv\|-\|u\|}{t}\quad\hbox{ and }\quad D_+(u;v):=\lim_{t\searrow 0}\frac{\|u+tv\|-\|u\|}{t}.
$$
By~\cite[Theorem~3.2]{James1}, we have $u \perp_{\RR} v$ if and only if $D_-(u;v) \leq 0 \leq D_+(u;v)$.  Now note that, by continuity of the norm,
\begin{align*}
D_\pm((x,0);(y,z)) &= \lim_{t\to \pm 0}\frac{\|(x + ty, tz)\|_{\cZ}-\|(x,0)\|_{\cZ}}{t}\\
&= \lim_{t\to \pm 0}\frac{\sqrt{\| x+ty \|_{\cX}^2 + \| tz \|_{\cY}^2}-\|x\|_{\cX}}{t}\\
&= \lim_{t\to \pm 0}\frac{\| x+ty \|_{\cX}^2 + t^2\| z \|_{\cY}^2-\|x\|_{\cX}^2}{t\left(\sqrt{\| x+ty \|_{\cX}^2 + \| tz \|_{\cY}^2} + \|x\|_{\cX}\right)}\\
&= \lim_{t\to \pm 0}\frac{(\| x+ty \|_{\cX}-\|x\|_{\cX}) (\| x+ty \|_{\cX} + \|x\|_{\cX})}{2 t \|x\|_{\cX}}\\
&= \lim_{t\to \pm 0}\frac{\| x+ty \|_{\cX}-\|x\|_{\cX}}{t} = D_\pm(x;y).
\end{align*}
Hence $(x,0) \perp_{\RR} (y, z)$ if and only if $x \perp_{\RR} y$. In the complex case, we use a general result that $u\perp v$ if and only if $u\perp_{\RR}\lambda v$ for all $\lambda \in \CC$.  Then $(x,0) \perp (y,z)$ if and only if  $(x,0) \perp_\RR \lambda (y,z)=(\lambda y,\lambda z)$ for all $\lambda \in \CC$ which, by above,  is equivalent to $x\perp_\RR \lambda y$ for all $\lambda \in \CC$, and then to $x\perp y$.

For the second part of the statement, note that $(y,z) \perp (x,0)$ if and only if $\sqrt{\| y \|_{\cX}^2 + \| z \|_{\cY}^2} = \| (y, z) \|_{\cZ} \leq \| (y, z) + \lambda (x,0) \|_{\cZ} = \sqrt{\| y + \lambda x \|_{\cX}^2 + \| z \|_{\cY}^2}$ for all $\lambda \in \FF$. Clearly, this is equivalent to $\| y \|_{\cX} \leq \| y + \lambda x \|_{\cX}$ for all $\lambda \in \FF$, that is, to  $y \perp x$.
\end{proof}

\begin{ex} \label{example:nonsmooth-point}
Consider $\cX = (\FF^2, \norm_{\mathcal R})$ from Example~\ref{example:nonsmooth-Radon-plane} or~\ref{example:complex-Radon-plane}, $\cY = (\FF^n, \norm_2)$ with $n \in \mathbb{N}$, and $\cZ = \cX \oplus_2 \cY$. Then $\dim \cZ = 2 + n \geq 3$.

We denote $\tilde{w} = (w, 0) \in \cZ$ for any $w \in \cX$. Since $\tilde{x} = (x,0)$ is a nonsmooth point in $\cX \oplus \{ 0 \}$, it is also a nonsmooth point in $\cZ$.

Let $e_1, \dots, e_n$ be the standard basis in $\cY$, and consider the vectors $x_1 = (0, e_1)$, \dots, $x_n = (0, e_n)$. By Lemma~\ref{lemma:direct-sum}, $x_j^\bot = \cX \oplus e_j^\bot$ and $\tilde{x}^\bot = x^\bot \oplus \cY$, where $x^\bot = \{ \alpha z ; \; \alpha \in \FF, \, z \in I \}$. Hence
$$
\tilde{x}^\bot \cap x_1^\bot \cap \dots \cap x_n^\bot = x^\bot \oplus \{ 0 \} = \FF I \oplus \{ 0 \}.
$$
Again by Lemma~\ref{lemma:direct-sum}, for any nonzero $z \in \FF I$ we have $\tilde{z}^\bot = z^\bot \oplus \cY = \FF x \oplus \cY = {}^\bot z \oplus \cY = {}^\bot \tilde{z}$, so $\FF I \oplus \{ 0 \}$ is a BJ-set.

Without loss of generality we can consider only the identity permutation of $x_1, \dots, x_n$. Let
$$
\Omega=\{w\in\Gamma(\cX);\;\; \tilde{x}^\bot\cap w^\bot\cap x_2^\bot \cap \dots\cap x_{n}^\bot \hbox{ is a BJ-set}\}.
$$
Clearly, $x_1 \in \Omega$. Besides, for any $z \in I$ we have $\tilde{x}^\bot\cap \tilde{z}^\bot\cap x_2^\bot \cap \dots\cap x_{n}^\bot = \FF x_1$ which is also a BJ-set, so $\tilde{z} \in \Omega$. Thus
$$
\tilde{x}^\bot \cap x_2^\bot\cap\dots\cap x_n^\bot \cap\bigcap\limits_{w\in\Omega}w^\bot \subseteq \tilde{z}^\bot \cap (\FF I \oplus \{ 0 \}) =\{0\},
$$
and the condition~(ii) of Corollary~\ref{corollary:high-dimensional} is satisfied for a nonsmooth point $\tilde{x}$.
\end{ex}

\section{Related results and concluding remarks}\label{section5}

Let $\cX$ be a normed space over $\FF$ and $x\in\cX$. Let us recall a few facts about the norm's subdifferential set as defined in Section \ref{section3}.

\begin{enumerate}
    \item[(i)] The subdifferential set is also related to the right hand derivative of the norm as follows (see \cite[Theorem 1.2.9]{Grover}):
    \begin{align}
    \partial\|x\| &{}=\{f\in \cX^*; \; \text{Re} f(y)\leq D_+(x, y)\text{ for all } y\in \cX\},\label{eq4}\\
    D_+(x, y) &{}= \max\{\text{Re}f(y); \; f\in \partial \|x\|\}.\label{eq2}
    \end{align}
    \item[(ii)] The subdifferential set $\partial\|x\|$ is a convex subset of $\cX^*$ (because convex sum of two supporting functionals is a supporting functional).
    \item[(iii)] The point $x \in \cX$ is smooth if and only if $\partial\|x\|$ is a singleton set and only contains the functional $f$ defined as $f(y) =D_+(x, y)=D_-(x, y)$ (see \cite[Corollary 1.2.10]{Grover}).
\end{enumerate}

 Now we give a characterization for the condition $x^\bot\subseteq y^\bot$, which complements Lemma \ref{thm5}.

\begin{prop} Let $\cX$ be a normed space over $\FF$ and let $x,y\in \cX$. Then \begin{align*}x^\bot\subseteq y^\bot & \iff  \big( \exists \alpha\in\mathbb F\setminus\{0\} \text{ such that } \partial \| \alpha x\|\subseteq \partial\| y\|\big)\\&\iff \big( \exists \alpha\in\mathbb F\setminus\{0\} \text{ such that } D_+(\alpha x, z)\leq D_+(y,z)\;\; \forall z\in \cX\big).\end{align*}
\end{prop}
\begin{proof} We first prove the equivalence: $$x^\bot\subseteq y^\bot  \iff  \big( \exists \alpha\in\mathbb F\setminus\{0\} \text{ such that } \partial \|\alpha x\|\subseteq \partial\| y\|\big).$$ The converse implication follows from \eqref{eq36} and the fact that $y^\bot=(\alpha y)^\bot$.

Now, let $x^\bot\subseteq y^\bot$. Without loss of generality we may assume that $\|x\|=\|y\|=1$. Using \eqref{eq36}, we obtain $x^\bot =\bigcup\{\Ker f;\; f\in \partial\|x\|\}$. Let $f$ be a supporting functional at $x$. Then $x^\bot\subseteq y^\bot$ implies $\Ker f\subseteq y^\bot$. It follows from~\cite[Theorem~2.1]{James1} that $|f(y)|=1$, i.e., there exists unimodular $\alpha$ such that $f(y)=\alpha$.

Hence for $f,g\in\partial\|x\|$ there exist unimodular $\alpha, \beta$  such that $f(y)=\alpha$ and $g(y)=\beta$. Since $(f+g)/2$ also belongs to $\partial\|x\|$,  there exists unimodular  $\gamma$ such that
$$\tfrac{\alpha+\beta}{2}=\left(\tfrac{f+g}{2}\right)(y) =\gamma.$$
Thus $|\alpha+\beta|=2$, so it follows from $|\alpha|=|\beta|=1$ that $\alpha=\beta$, and hence $\partial\| \alpha x\|\subseteq \partial\| y\|$. The last equivalence follows directly from Eqs.~\eqref{eq4} and~\eqref{eq2}.
\end{proof}

\begin{cor} Let $\cX$ be a normed space and let $x,y\in \cX$. We have \begin{align*}x^\bot = y^\bot & \iff  \big( \exists \alpha\in\mathbb F\setminus\{0\} \text{ such that } \partial \|\alpha x\|= \partial\|y\|\big)\\&\iff \big( \exists \alpha\in\mathbb F\setminus\{0\} \text{ such that } D_+(\alpha x, z)= D_+(y,z)\;\;\forall z\in \cX\big).\end{align*}
\end{cor}
By Lemma~\ref{thm5}, two smooth vectors belong to the same face if and only if they have the same supporting functional, which, in finite-dimensional spaces,  coincides with their gradients. This yields the next observation.
\begin{cor}
Two smooth normalized vectors in a finite-dimensional normed space lie on the same face of the norm’s unit ball if and only if their gradients are equal.
\end{cor}

We show next that there do exist nonsmooth and nonrotund real BJ-norms. The following example proves the existence of such a norm in $\RR^3$, but it can be easily generalized to $\RR^n$ with $n \geq 3$.
\begin{ex} \label{example:BJ-norm}
Consider
$$\|(x,y,z)\|:=\begin{cases}
 \sqrt{2} \sqrt {x^2+y^2+z^2}, & |z|\leq \sqrt {x^2+y^2}; \\
 \sqrt {x^2+y^2} +|z|, & \text{otherwise.} \end{cases}$$
To see that this is indeed a norm, we only need to prove the triangle inequality. This will be done  indirectly by studying the set of points $C$ where $\norm\le1$. Notice first that, by Cauchy-Schwarz-Bunyakowsky inequality, applied to $(\sqrt{x^2+y^2},|z|)$ and $(1,1)$, we immediately have
\begin{equation}\label{eq:smmothoutcone}
\sqrt{x^2+y^2}+|z|\le\sqrt{2}\sqrt{x^2+y^2+z^2},
\end{equation}
and the equality holds if and only if $|z|=\sqrt{x^2+y^2}$. Therefore, the set $C$ coincides with a cone $\{(x,y,z)\in\RR^3;\;\sqrt{x^2+y^2}+|z|\le1\}$, when $|z|\ge \sqrt{x^2+y^2}$, and a Euclidean ball of radius $1/\sqrt{2}$, otherwise. By \eqref{eq:smmothoutcone}, this ball is inscribed into the cone, so it touches it tangentially, and as such the function $\norm$ is smooth at nonzero points of intersection $|z|=\sqrt{x^2+y^2}$. Also, $C$ is invariant under rotation around the $z$-axis, more precisely, under the action of a matrix $U\oplus 1$ where $U$ is an orthogonal $(2 \times 2)$  matrix. Hence $C$ coincides with a rotational body obtained by rotating
$C\cap(\{0\}\times(-\infty,0]\times \RR)$ around the $z$-axis.  One can compute that the boundary points of this intersection lie on a convex  curve
$$
y(z) = \begin{cases}
    -\sqrt{\frac{1}{2}-z^2}, & \phantom{\frac{1}{2}<{}} |z| \leq \frac{1}{2}; \\
   |z|-1, & \frac{1}{2}<{} |z|\leq 1.
\end{cases}
$$
As such,
$$C=\left\{\left(y \cos\theta, y \sin\theta,z \right);\;\;0\le \theta\le 2\pi,\; y \leq 0, \; \|(0,y,z)\|\le 1\right\},$$
being a rotation of a convex planar region, is also convex. Moreover,  $C$ is clearly a compact, absorbing,  and balanced set whose boundary points  lie outside the open ball of radius $1/\sqrt{2}$. So, by~\cite[Theorem~1.35]{RUD}, its Minkowsky functional  is a norm whose closed unit ball coincides with $C$. Then Minkowsky functional must be equal to $\norm$, and the triangle inequality follows.

We proceed to show that $\norm$ is a BJ-norm. It suffices to verify that there exist no linearly independent normalized vectors $u,v\in\RR^3$ such that $u^\bot=v^\bot$ and ${}^\bot u={}^\bot v$. Assume otherwise.  By Lemma~\ref{thm5}, after replacing, if needed, $v$ with $-v$ we achieve that $u$  and $v$ belong to the same norm's face, i.e., the same face of norm's unit ball $C$.
Clearly, neither $u$ nor $v$ can coincide with the only two nonsmooth points $(0,0,\pm1)$ of $C$ because then one among $u^\bot,v^\bot$ would be a hyperplane, while  the other would be a union of several distinct hyperplanes, contradicting $u^\bot=v^\bot$.
Hence both $u,v$ are smooth and belong to the same face. Since $u^\bot=\Ker f_u$, where $f_u$ is a supporting functional at $u$ which can be identified with the norm's gradient at $u$, we see that $\partial\|u\|$ and $\partial\|v\|$ are parallel. Now, the norm's gradient at a smooth point $(x,y,z)$ is easily seen to be
$$\begin{cases}
\dfrac{\sqrt{2}(x,y,z)}{\sqrt{x^2+y^2+z^2}}, & | z| \leq \sqrt{x^2+y^2}; \\
   \left(\dfrac{x}{\sqrt{x^2+y
   ^2}},\dfrac{y}{\sqrt{x^2+y^2   }},\mathrm{sign}\, z\right), &  \text{otherwise}.
\end{cases}
$$
Therefore, under our assumption that $u, v$ are two distinct points which belong to the same face, the two gradients at $u=(u_1,u_2,u_3)$ and $v=(v_1,v_2,v_3)$ are parallel if and only if $|u_3| \ge \sqrt{u_1^2+u_2^2}$ and $|v_3| \ge \sqrt{v_1^2+v_2^2}$ and $(u_1,u_2)$ is parallel to $(v_1,v_2)$. By a suitable rotation around  the $z$-axis  and multiplication by $\pm1$ (which are both  isometries of our norm, hence preserve the in-coming and out-going neighbourhoods), we can achieve that normalized and smooth points $u,v$ satisfy $u=(0,u_2,1-u_2)$ and $v=(0,v_2,1-v_2)$ with $0< u_2,v_2\le 1/2$. The only smooth points in ${}^\bot u$ which lie on the hyperplane $\Pi=\{0\}\times\RR^2$ are the ones whose supporting functional annihilates $u$. Since supporting functional is given by the norm's gradient, we hence see that  the only such points are $(0,\lambda(1- 1/u_2),\lambda)$ where $\lambda\in\RR$. They clearly do not belong to ${}^\bot v\cap\Pi$, therefore, ${}^\bot u\neq{}^\bot v$. Thus we have shown that ${}^\bot u={}^\bot v$ and $u^\bot=v^\bot$ for some nonzero $u,v$  if and only if $\RR u=\RR v$.
\end{ex}

A norm on a finite-dimensional real vector space is called polyhedral if its closed unit ball is the intersection of finitely many closed halfspaces. This is equivalent to having only finitely many faces (see, e.g., \cite[Theorem 19.1]{Rockafellar}  for an even more general result).

\begin{prop}\label{prop:polyhedral}
Let $(\cX, \norm)$ be a real finite-dimensional normed space. Then the norm $\norm$ is polyhedral if and only if $\big|\{x^\bot;\; x\in \Gamma(\cX)\}\big|<\infty$.
\end{prop}
\begin{proof} If the norm $\norm$ is polyhedral, then it has finitely many faces. Using Remark~\ref{substitute:lemma2.10}, we obtain $x^\bot=y^\bot$ for all $x,y$ in the relative interior of any particular face. By~\cite[Theorem 18.2]{Rockafellar}, the norm's unit sphere is a disjoint union of relative interiors of its faces, and hence $\big|\{x^\bot;\; x\in \Gamma(\cX)\}\big|<\infty$.

Conversely, assume that $|\{x^\bot;\; x\in \Gamma(\cX)\}|=m$. Let $x_1, \dots, x_m$ be points of $\Gamma(\cX)$ such that $x_i^\bot\neq x_j^\bot$ for all $i\neq j$. Since for any $y\in\Gamma(\cX)$ we have $y^\bot=x_i^\bot$ for some $1\leq i\leq m$, Lemma \ref{thm5} implies that $y$ and $x_i$ belong to the same face. Therefore, the number of faces of the unit ball is less than equal to $m$. Hence the norm $\norm$ is a polyhedral norm.\end{proof}

\bigskip

Finally, let us conclude with some observations.

\begin{itemize}
    \item[(A)] Recall that a complexification of a real normed space $(\cX, \norm)$ is a complex vector space $\cX_{\mathbb C} = \mathbb C\otimes_{\mathbb R}\cX$ equipped with a complex norm $\norm_{\mathbb C}$ that agrees with $\norm$ on the real subspace $1\otimes_{\mathbb R}\cX$. Our next example shows that there might exist different complexifications $(\cY_1, \norm_1)$ and $(\cY_2, \norm_2)$ of $(\cX, \norm)$ such that projective orthodigraphs $\Gamma(\cY_1)$ and $\Gamma(\cY_2)$ are not  isomorphic. For example, consider two complexifications of the real Euclidean space $(\mathbb R^n, \norm_2)$:   the complex Hilbert space $(\mathbb C^n, \norm_2)$ and the Taylor minimal complexification $(\mathbb C^n, \norm_{\mathfrak T})$. These two complexifications of  $(\mathbb R^n, \norm_2)$ are not isometric because their isometry groups are not isomorphic (see \cite[Example 6.1]{Li}). By James' result, BJ-orthogonality in   $(\mathbb C^n, \norm_{\mathfrak T})$ is not symmetric, so  projective orthodigraphs $\Gamma(\mathbb C^n, \norm_2)$ and $\Gamma(\mathbb C^n, \norm_{\mathfrak T})$ are not  isomorphic.

\item[(B)]\label{remark:exposed}
Corollary \ref{cor:maximalfaceclassify} holds even if we replace the word `maximal face' by `maximal exposed face'. This follows from the fact that, in a finite-dimensional normed space, a maximal face is always exposed. To see this, assume that $F$ is the norm's maximal face. Let $A=\Span_{\mathrm{aff}} F$ be its (real) affine span. If $F$ is zero-dimensional, then let $z \in F$ be its unique point. Otherwise, let $z$ be an arbitrary point from the (nonempty) relative interior of $F$. Consider a linear subspace $\cV = A-z$, so that $A = \cV + z$. Since either $\cV = 0$ or $z$ belongs the relative interior of $F$, for any $v \in \cV$ we have $z + \lambda v \in F$ for $\lambda \in \mathbb{R}$ with small enough $|\lambda|$, see~\cite[Corollary~6.4.1]{Rockafellar}. But $\| x \| = 1$ for all $x \in F$, so for any $v \in \cV$ it holds that $D_{\pm}(z;v) = 0$, and thus, by~\cite[Theorem~3.2]{James1}, $z \perp v$. It follows from~\cite[Theorem~2.1]{James1} that there exists an $\RR$-linear functional $f$ with $\| f \| = 1$ such that $f(z) = 1$ and $\cV \subseteq \Ker f$. Then $f(F)\subseteq f(A)=f(\cV+z)=\{1\}$, so $F\subseteq f^{-1}(1)\cap \boldsymbol{B}_{\cX}$, where $\boldsymbol{B}_{\cX}$ is the norm's closed unit ball of $\cX$. Since this intersection is also a face, and~$F$, being a maximal face, cannot be properly contained in another face, it follows that $F= f^{-1}(1)\cap \boldsymbol{B}_{\cX}$. Thus $F$ is an exposed face.

\item[(C)] In particular, each maximal face in a  complex Radon  plane~$\mathcal{R}$  takes the form $f^{-1}(1)\cap\boldsymbol{B}_{\mathcal{R}}$ for  some normalized $\CC$-linear functional $f$, so its  dimension is either $0$ or $2$. Thus, condition (ii) of Lemma~\ref{lem:nnnsmoothRadonplaneCC} is also equivalent to:
(a$'$) The boundary points of each two-dimensional face in ${\mathcal R}$ are smooth,~and (b$'$) the set $\{x^\bot; \;\;x\in\Gamma_0({\mathcal R})\}$ has continuum many points.
\item[(D)] Note that in the proof of Lemma \ref{smooth-rotund} we showed that for a smooth norm the condition ${}^\bot x= {}^\bot y$ implies $\FF x = \FF y$. With the out-going neighborhoods we can say more: in \cite[Lemma~2.6]{AGKRZ23} it was proved that a norm on $\cX$ is strictly convex if and only if, given $x, y \in \cX \setminus \{ 0 \}$, the condition $x^\bot = y^\bot$ implies $\FF x = \FF y$.

\item[(E)]One way to interpret the existence of an isomorphism in item (i) of Lemmas~\ref{lem:nnnsmoothRadonplaneRR} and~\ref{lem:nnnsmoothRadonplaneCC} (and the existence of appropriate nonsmooth Radon planes as guaranteed by Examples~\ref{example:nonsmooth-Radon-plane} and~\ref{example:complex-Radon-plane}) is that  orthodigraph $\Gamma_0$ cannot detect nonsmooth norms on two-dimensional spaces. In this regard, its projective counterpart, orthodigraph $\Gamma$   is again more restrictive: by Lemma~\ref{lemma:two-dimensional-projective}, there can be no isomorphism  between $\Gamma({\mathcal R})$ for a (real or complex)  nonsmooth Radon plane and $\Gamma(\FF^2,\norm_2)$ for a (real or complex) two-dimensional Hilbert space.

\item[(F)] Let ${\mathcal R}$ be a smooth real Radon plane which is not isometrically isomorphic to a two-dimensional Hilbert space $\ell_2^2$, and, for a nonempty index set $I$, let  $c_0(I)$  denote the Banach space of  all systems $(a_n)_{n\in I}$ such
that $\{n \in I;\; |a_n| \ge \varepsilon \}$  is finite for each $\varepsilon > 0$. In a recent paper, Tanaka~\cite[Theorem 3.13]{Tanaka22a} showed that Banach spaces ${\mathcal R}\oplus_{\infty} c_0(I)$ and $\ell_2^2\oplus_{\infty} c_0(I)$ are not isometrically isomorphic, but still there does exist a homogeneous bijection $\Phi\colon {\mathcal R}\oplus_{\infty} c_0(I)\to \ell^2_2\oplus_{\infty}  c_0(I)$, which preservers BJ-orthogonality in both directions. Because it is homogeneous, it hence induces a graph isomorphism between projective orthodigraphs  $\Gamma({\mathcal R}\oplus_{\infty }c_0(I))$ and $\Gamma(\ell_2^2\oplus_{\infty} c_0(I))$. This shows that, in nonsmooth normed spaces,  orthodigraphs cannot completely classify the norms up to isometry.

\item[(G)] Shortly after the submission of this paper we were informed that Saikat Roy and Ryotaro Tanaka just recently obtained yet another classification, parallel to our  Lemma~\ref{lem:nnnsmoothRadonplaneRR},  of real Radon planes with $\Gamma_0({\mathcal R})\simeq \Gamma_0(\RR^2,\norm_2)$  (see~\cite[Corollary 2.14 and Theorems~3.10 and~3.13]{Roy-Tanaka}). Moreover, \cite[Example~3.16]{Roy-Tanaka} supplements our Example \ref{example:absolute-Radon-plane} of a nonsmooth and nonrotund real Radon plane $\mathcal R$ which satisfies this condition. However, to the best of the authors' knowledge, our Lemma \ref{lem:nnnsmoothRadonplaneCC} and the accompanying Example~\ref{example:complex-Radon-plane}, for the complex Radon planes are completely new.

\item[(H)] For a complex normed space $\cX$, if we define two norm's faces $F_1$ and~$F_2$ to be equivalent if there exists a unimodular scalar $\mu$ such that $F_1=\mu F_2$, then the result for complex normed space corresponding to Proposition \ref{prop:polyhedral} states as follows. There exists finitely many nonequivalent norm's faces if and only if $\big|\{x^\bot;\; x\in\Gamma(\cX)\}\big|< \infty.$ We also remark that this characterization is  valid even in nonprojective case. More precisely, there exists finitely many nonequivalent norm's faces in a finite-dimensional normed space $(\cX, \norm)$ if and only if $\big|\{x^\bot;\; x\in \Gamma_0(\cX)\}\big|<\infty$. This follows directly because from \eqref{eq15}, we see that there is bijective correspondence between $\{x^\bot;\; x\in \Gamma(\cX)\}$ and $\{x^\bot;\; x\in \Gamma_0(\cX)\setminus\{0\}\}$ given by $[x]^\bot\longleftrightarrow x^\bot.$
\end{itemize}

\end{document}